\documentclass[a4paper]{amsart}
\usepackage[american]{babel}
\usepackage{amsmath,amsfonts,amssymb,amsthm,bbm}
\usepackage{graphicx}

\newcommand{\R}{\mathbb{R}}

\newcommand{\V}{\mathbf{V}}
\newcommand{\J}{\mathbf{J}}

\newcommand{\sech}{\textrm{sech}}

\newcommand{\dn}{{\text{\tiny dn}}}

\newcommand{\up}{{\text{\tiny up}}}

\newcommand{\zO}{\Omega}
\newcommand{\im}{\mathbf{i}}
\newcommand{\ftran}[1]{\mathcal{F}\left\{{#1}\right\}}

\newcommand{\iftran}[1]{\mathcal{F}^{-1}\left\{{#1}\right\}}

\theoremstyle{remark}\newtheorem{remark}{Remark}
\theoremstyle{plain}\newtheorem{theorem}{Theorem}[section]
\theoremstyle{plain}
\numberwithin{equation}{section}

\usepackage{framed}
\newenvironment{tightcenter}{\setlength\topsep{5pt}\setlength\parskip{-5pt}\begin{center}}{\end{center}}

\graphicspath{{./}{figure/}}

\title{Scattering data computation for the Zakharov-Shabat system}
\author{L. Fermo, C. van der Mee and S. Seatzu}
\address{Department of Mathematics and Computer Science\\
		University of Cagliari \\ Viale Merello 92, 09123 Cagliari, Italy}

\begin{document}

\maketitle

\begin{abstract}
A numerical method to solve the direct scattering problem for the Zakharov-Shabat system associated to the initial value problem for the nonlinear Schr\"odinger equation is proposed. The method involves the numerical solution of Volterra integral systems with structured kernels  and the identification of coefficients and parameters appearing in monomial-exponential sums. Numerical experiments confirm the effectiveness of the proposed technique.

\medskip

\noindent{\bf Keywords:} Nonlinear Schr\"o\-din\-ger equation, Inverse Scattering Transform, Integral equations

\medskip

\noindent{\bf Mathematics Subject Classification:} 41A46, 65R20, 35P25
\end{abstract}

\section{Introduction}
The problem we are addressing concerns the numerical computation of the scattering data of the Zakharov-Shabat (ZS) system associated to the initial value problem (IVP) for the nonlinear Schr\"o\-din\-ger (NLS) equation 
\begin{equation}\label{NLS}
\begin{cases}
{\im} u_t+u_{xx}\pm 2 |u|^2u=0, \quad x\in\R, \quad t>0 \\
u(x,0)=u_0(x), \quad x \in \R
\end{cases}
\end{equation}
where  $\im$ denotes the imaginary unit, $u=u(x,t)$ is the unknown potential, the subscripts $x$ and $t$ designate partial derivatives with respect to position and time, $u_0 \in L^1(\R)$ is the initial potential and the $\pm$ sign depends on symmetry properties of $u$. The plus sign regards the focusing case and the minus sign the defocusing case.

The solution of the IVP \eqref{NLS} can theoretically be obtained by means of the so-called Inverse Scattering Transform (IST) technique \cite{AC,AS}. The IST allows one, in fact, to obtain the solution of \eqref{NLS} by means of the following three steps:
\begin{itemize}
\item[(i)] starting from the initial potential $u_0$, solve the  Zakharov-Shabat (ZS) system associated to the NLS to obtain the initial scattering data;
\item[(ii)] propagate the initial scattering data in time;
\item[(iii)] solve the associated Marchenko equations whose kernels are obtained from the initial scattering data evolved in time, to obtain the solution $u(x,t)$ we are looking for.  
\end{itemize}
An effective numerical method to solve steps (ii) and (iii) has been proposed in \cite{ArRoSe2011} under the hypothesis that the initial scattering data are known. In this paper we propose a numerical method to solve the direct scattering problem (i) which is also of independent interest in some engineering fields \cite{Wahls2013}. To the best of our knowledge our method is the first numerical method proposed for the computation of all scattering data.

The paper is organized as follows. In Section \ref{sect:initial} we recall the ZS system associated to the IVP for the NLS equation. Then we recall the definition of the initial scattering data, i.e. the transmission coefficient $T(\lambda)$, the reflection coefficient from the left $L(\lambda)$ and from the right $R(\lambda)$, the bound states  $\{ \lambda_j \}$ with their multiplicities $\{ m_j\}$ and the norming constants from the left $\{(\Gamma_{\ell})_{j,s}\}$ and from the right $\{(\Gamma_{r})_{js}\}$. After that, we introduce the initial Marchenko kernels from the left $\Omega_\ell(\alpha)$ and from the right $\Omega_r(\alpha)$, the inverse Fourier transform $\rho(\alpha)$ of $R(\lambda)$ and the Fourier transform $\ell(\alpha)$ of $L(\lambda)$, respectively. Then, we show that the spectral sums from the left $S_\ell(\alpha)$ and from the right $S_r(\alpha)$ which depend on the bound states with the respective multiplicities and  norming constants from the left and from the right, can be expressed as a difference between the initial Marchenko kernels and the inverse Fourier transform $\rho(\alpha)$ of $R(\lambda)$ and the Fourier transform $\ell(\alpha)$ of $L(\lambda)$, respectively. As these differences are monomial-exponential sums, their parameters and coefficients can be identified by using the numerical method proposed in \cite{Fermo2014-AMC,Fermo2014-ETNA}. Section \ref{sect:auxiliary} is devoted to the characterization of auxiliary functions which are basic to the computation of the initial Marchenko kernels as well as of $\rho(\alpha)$ and $\ell(\alpha)$. In Section \ref{sect:Marchenko} we characterize the scattering matrix,  derive and analyze the Volterra integral equations of the second kind that characterize the initial Marchenko kernels and formulate the Fredholm integral equations that characterize the Fourier transforms $\rho(\alpha)$ and $\ell(\alpha)$. 
The numerical method  we propose to obtain the initial scattering data is illustrated in Section \ref{sect:nummethod} while in Section \ref{sect:examples} we consider two different initial potentials for which the numerical results are given in Section \ref{sect:tests}. Finally, we conclude the paper by an Appendix concerning the  study of the support of the auxiliary functions introduced in Section \ref{sect:auxiliary}.

\section{Initial scattering data}
\label{sect:initial}
Following the IST technique, to determine the initial scattering data, we must consider the ZS system associated to the NLS \eqref{NLS} \cite{Ablowitz2004}, that is the system 
\begin{equation}\label{ZS}
{\im} \J \frac{\partial \Psi}{\partial x}(\lambda,x)- \V(x) \Psi(\lambda,x)=\lambda \Psi(\lambda,x), \quad x \in \R
\end{equation}
where $\lambda \in \mathbb{C}$ is a spectral parameter and
\begin{equation}\label{J}
 \J=\begin{pmatrix}1&0\\ 0&-1\end{pmatrix},\qquad  \V=\im \begin{pmatrix}0& u_0\\
v_0 & 0\end{pmatrix} 
\end{equation}
with $v_0=u_0^*$ in the focusing case and $v_0=-u_0(x)^*$ in the defocusing case. Here and in the sequel the asterisk denotes the complex conjugate. 

The initial scattering data are the entries of the so-called scattering matrix and the coefficients and parameters of two spectral sums.
Denoting by
\begin{equation*}\label{S}
{\bf S}(\lambda)=\begin{pmatrix}T(\lambda)&L(\lambda)\\ R(\lambda)& T(\lambda)
\end{pmatrix},
\end{equation*}
the  scattering matrix, $T(\lambda)$ represents  the (initial) transmission coefficient, while $L(\lambda)$ and $R(\lambda)$ stand for the  initial reflection coefficients from the left and from the right, respectively. This matrix satisfies the following symmetry properties \cite{VanDerMee2013}
\begin{equation}\label{proprS1}
{\bf S^\dagger}(\lambda) {\bf{S}}(\lambda)={\bf{S}}(\lambda) {\bf{S}^\dagger}(\lambda)={\bf{I}}
\end{equation}
in the defocusing case and 
\begin{equation}\label{proprS2}
{\bf S^\dagger}(\lambda) \, {\bf{J}} \, {\bf{S}}(\lambda)={\bf{S}}(\lambda) \, {\bf{J}} \, {\bf S^\dagger}(\lambda)={\bf{J}}
\end{equation}
in the focusing case where ${\bf{I}}$ denotes the identity matrix. Here and in the sequel the dagger denotes the matrix conjugate transpose.
The numerical validity of these properties is used in Section \ref{sect:tests} to check the effectiveness of our algorithms.

If $T(\lambda)$ has no poles in the complex upper half plane $\mathbb{C}^+$, there are no spectral sums to identify. Otherwise, denoting by $\lambda_1,\, \dots,\, \lambda_n$ the so-called bound states, that is the finitely many poles of $T(\lambda)$ in $\mathbb{C}^+$,  and by $m_1,\,\dots,\, m_n$ the corresponding multiplicities, we have to identify the parameters $\{n,m_j,\lambda_j \}$ as well as the coefficients $\{(\Gamma_\ell)_{js}, (\Gamma_r)_{js}\}$ of the initial spectral sums from the left and from the right
\begin{align}
S_\ell(\alpha)=\sum_{j=1}^{ n} e^{i \lambda_j \alpha} \sum_{s=0}^{{m}_j-1}( \Gamma_\ell)_{js} \frac{\alpha^s}{s!}, \quad \alpha \geq 0 \label{S_ell}\\
S_r(\alpha)=\sum_{j=1}^n e^{i \lambda^*_j \alpha} \sum_{s=0}^{m_j-1}(\Gamma_r)_{js} \frac{\alpha^s}{s!}, \quad \alpha \leq 0 \label{S_r}.
\end{align} 
In \eqref{S_ell} and \eqref{S_r} the coefficients $(\Gamma_\ell)_{js}$ and $(\Gamma_r)_{js}$ are the so-called  norming constants from the left and from the right, respectively, and $0! = 1$.

In the IST technique, a crucial role is played by the initial Marchenko kernels from the left $\Omega_\ell(\alpha)$ and from the right $\Omega_r(\alpha)$, which are connected to the above spectral coefficients and spectral sums as follows: 
\begin{align}
\zO_\ell(\alpha)&=\rho(\alpha)+S_\ell(\alpha), \quad \textrm{for} \quad \alpha \geq 0 \label{Omega_ell} \\ \nonumber \\
\zO_r(\alpha)&=\ell(\alpha)+S_r(\alpha), \quad \textrm{for} \quad \alpha \leq 0 \label{omega_r} 
\end{align}
where 
\begin{align}\label{rho}
\rho(\alpha)= \frac{1}{2\pi}\int_{-\infty}^{+\infty} R(\lambda) e^{i \lambda \alpha} d \lambda =\iftran{R(\lambda)}
\end{align}
is the inverse Fourier transform of the reflection coefficient from the right $R(\lambda)$ 
and 
\begin{align}\label{elle}
\ell(\alpha)= \frac{1}{2\pi}\int_{-\infty}^{+\infty} L(\lambda) e^{-i \lambda \alpha} d \lambda =\frac{1}{2\pi} \ftran{L(\lambda)},
\end{align}
apart from the factor $1/2\pi$, is the Fourier transform of the reflection coefficient from  the left $L(\lambda)$.

We note that $\Omega_\ell(\alpha)$ and $\Omega_r(\alpha)$, respectively, reduce to:
\begin{itemize}
\item[(a)] $S_\ell(\alpha)$ and $S_r(\alpha)$ if the reflection coefficients vanish (reflectionless case);
\item[(b)] $\rho(\alpha)$ and  $\ell(\alpha)$ if there are no bound states.
\end{itemize}
\section{Auxiliary functions}\label{sect:auxiliary}
In this section we introduce four pairs of  auxiliary functions and the Volterra integral equations that characterize them. Their solution, as shown in the next section (see also \cite{Fermo2014-PUB,VanDerMee2013}), is fundamental for computing the initial Marchenko kernels as well as $\rho(\alpha)$ and $\ell(\alpha)$.

Following \cite{Fermo2014-PUB}, let us introduce,  for $y \geq x$,  the two pairs of unknown auxiliary functions
$$\bar{ {\bf K}}(x,y) \equiv\begin{pmatrix}\bar{K}^{\up}(x,y)\\ \bar{K}^{\dn}(x,y) \end{pmatrix}, \quad { {\bf K}}(x,y) \equiv\begin{pmatrix}K^{\up}(x,y)\\ K^{\dn}(x,y) \end{pmatrix},$$ \\
and, for $y \leq x$, the two pairs of unknown auxiliary functions 
$$\bar{ {\bf M}}(x,y) \equiv\begin{pmatrix}\bar{M}^{\up}(x,y)\\ \bar{M}^{\dn}(x,y) \end{pmatrix}, \quad { {\bf M}}(x,y) \equiv\begin{pmatrix}M^{\up}(x,y)\\ M^{\dn}(x,y) \end{pmatrix}.$$

For the sake of clarity, let us explain how these functions are connected to the Jost matrices associated to the ZS system \eqref{ZS}. 

As in \cite{Ablowitz2004,VanDerMee2013}, we represent the Jost matrices as the Fourier transforms of the auxiliary functions:
 
\begin{align}
(\bold{\Psi}(\lambda,x), \bold{\bar{\Psi}}(\lambda,x)) = e^{-i\lambda \bold{J} x} +\int_{x}^\infty ( {\bold{K}}(x,y), {\bold{\bar{K}}}(x,y)) \, e^{-i\lambda \bold{J} y} \, dy, \label{JostPsi}\\
(\bold{\Phi}(\lambda,x), \bold{\bar{\Phi}}(\lambda,x)) = e^{-i\lambda \bold{J} x} +\int_{-\infty}^x ( {\bold{M}}(x,y), {\bold{\bar{M}}}(x,y)) \, e^{-i\lambda \bold{J} y} \, dy \label{JostPhi},
\end{align}
from which inverting the Fourier transforms we  get
\begin{align}
( {\bold{K}}(x,y), {\bold{\bar{K}}}(x,y))= \frac{1}{2\pi} \int_{-\infty}^\infty [(\bold{\Psi}(\lambda,x), \bold{\bar{\Psi}}(\lambda,x))- e^{-i\lambda \bold{J} x}] e^{i\lambda \bold{J} y} \, d \lambda \label{KJost},\\
( {\bold{M}}(x,y), {\bold{\bar{M}}}(x,y))= \frac{1}{2\pi} \int_{-\infty}^\infty [(\bold{\Phi}(\lambda,x), \bold{\bar{\Phi}}(\lambda,x))- e^{-i\lambda \bold{J} x}] e^{i\lambda \bold{J} y} \, d \lambda \label{MJost}.
\end{align}
Now, for $y \geq x$, the pair $(\bar{K}^{\up}, \, \bar{K}^{\dn})$ is the solution of the following system of two structured Volterra integral equations \cite{Fermo2014-PUB,VanDerMee2013}: 
\begin{equation}\label{nucleiKbarrati}
\begin{cases}
\bar{K}^\up(x,y)+\displaystyle\int_x^\infty u_0(z) \, \bar{K}^\dn(z,z+y-x) \, dz=0\\ \\
\bar{K}^\dn(x,y)-\displaystyle\int_x^{\tfrac{1}{2}(x+y)} v_0(z) \, \bar{K}^\up(z,x+y-z) \, dz=\tfrac{1}{2} v_0(\tfrac{1}{2}(x+y))
\end{cases}
\end{equation}
while the pair $(K^{\up}, \, K^{\dn})$ is the solution of the system 
\begin{equation}\label{nucleiK}
\begin{cases}
K^\up(x,y)+\displaystyle\int_x^{\tfrac{1}{2}(x+y)}\,u_0(z)\, K^\dn(z,x+y-z) \, dz=-\tfrac{1}{2}u_0(\tfrac{1}{2}(x+y)) \\ \\
K^\dn(x,y)-\displaystyle\int_x^\infty \,v_0(z) \, K^\up(z,z+y-x) \, dz=0.
\end{cases}
\end{equation}

Similarly, for $y \leq x$ the pair $(\bar{M}^{\up}, \, \bar{M}^{\dn})$ is the solution of the system of two structured Volterra equations:
\begin{equation}\label{nucleiMbarrati}
\begin{cases}
\bar{M}^\up(x,y)-\displaystyle\int_{\tfrac{1}{2}(x+y)}^x u_0(z)\, \bar{M}^\dn(z,x+y-z) \, dz =\tfrac{1}{2}u_0(\tfrac{1}{2}(x+y)) \\ \\ 
\bar{M}^\dn(x,y)+\displaystyle\int_{-\infty}^x v_0(z) \, \bar{M}^\up(z,z+y-x) \, dz=0
\end{cases}
\end{equation}
and the pair $({M}^{\up}, \, {M}^{\dn})$ is the solution of the following system
\begin{equation}\label{nucleiM}
\begin{cases}
M^\up(x,y)-\displaystyle\int_{-\infty}^x u_0(z)\,M^\dn(z,z+y-x) \, dz=0 \\ \\
M^\dn(x,y)+\displaystyle\int_{\tfrac{1}{2}(x+y)}^x v_0(z) \,M^\up(z,x+y-z) \, dz=-\tfrac{1}{2} v_0(\tfrac{1}{2}(x+y)).
\end{cases}
\end{equation}

From the computational point of view, on the bisector $y=x$, it is important to note that  each auxiliary function is uniquely determined  by the initial solution or its partial integral energy. In fact, setting $y=x$ in each of the four Volterra systems, we immediately obtain:
\begin{align}
\bar{K}^\dn(x,x)&= \frac 1 2 v_0(x), \quad 
\bar{K}^\up(x,x)  =-\frac 12 \int_x^\infty u_0(z) \, v_0(z) \, dz, \label{propr1}\\
K^\up(x,x)&= -\frac 1 2   u_0(x), \quad 
K^\dn(x,x)=-\frac{1}{2}\int_x^\infty u_0(z) \, v_0(z) \, dz, \label{propr2} \\
M^\dn(x,x) &=-\frac 1 2   v_0(x),\quad  M^\up(x,x) =- \frac{1}{2} \int_{-\infty}^x u_0(z) \, v_0(z)\, dz, \label{propr3} \\
\bar{M}^\up(x,x)&= \frac 1 2 u_0(x),\quad 
\bar{M}^\dn(x,x)= - \frac{1}{2}\int_{-\infty}^x u_0(z) \, v_0(z) \, dz \label{propr4}.
\end{align}


Moreover, let us mention that the functions $\bold{\bar{K}}$ and $\bold{{K}}$, as well as the functions $\bold{\bar{M}}$ and $\bold{{M}}$, are related to each other. 
Indeed, in the focusing case the following symmetry properties hold true \cite{VanDerMee2013}  

\begin{equation} \label{symmpropr1}
\left( \begin{matrix}
K^{\up}(x,y) \\ \\  K^{\dn}(x,y)
\end{matrix} \right)= \left( \begin{matrix}
-\bar{K}^{\dn}(x,y)^*  \\  \\  \bar{K}^{\up}(x,y)^* 
\end{matrix} \right), \qquad 
\left( \begin{matrix}
M^{\up}(x,y) \\ \\  M^{\dn}(x,y)
\end{matrix} \right)=  \left( \begin{matrix}
\bar{M}^{\dn}(x,y)^*
\\ \\  -\bar{M}^{\up}(x,y)^* \end{matrix}
 \right)
\end{equation}
while in the defocusing case the following symmetry relations can be proved
\begin{equation} \label{symmpropr2}
\left( \begin{matrix}
K^{\up}(x,y) \\ \\  K^{\dn}(x,y)
\end{matrix} \right)= \left( \begin{matrix}
\bar{K}^{\dn}(x,y)^*  \\  \\  \bar{K}^{\up}(x,y)^* 
\end{matrix} \right), \qquad 
\left( \begin{matrix}
M^{\up}(x,y) \\ \\  M^{\dn}(x,y)
\end{matrix} \right)=  \left( \begin{matrix}
\bar{M}^{\dn}(x,y)^*
\\ \\  \bar{M}^{\up}(x,y)^* \end{matrix}
 \right).
\end{equation}

\begin{remark} \label{remark1}
Let us note that, in virtue of \eqref{symmpropr1}-\eqref{symmpropr2}, we only need to compute numerically systems \eqref{nucleiKbarrati} and \eqref{nucleiMbarrati} or systems \eqref{nucleiK} and \eqref{nucleiM} and then compute the remaining auxiliary functions by resorting to the above symmetry properties . 
\end{remark}

\begin{remark}\label{remark2}
If the potentials $u_0$ and $v_0$ are even functions, the auxiliary functions ${\bf{M}}$ can easily be obtained from the ${\bf{K}}$ functions as follows:
\begin{align*}
M^{\up}(x,y)&=\bar{K}^{\up}(-x,-y) \qquad M^{\dn}(x,y)=-\bar{K}^{\dn}(-x,-y) \\
\bar{M}^{\up}(x,y)&=-K^{\up}(-x,-y) \qquad \bar{M}^{\dn}(x,y)=K^{\dn}(-x,-y).
\end{align*}
Similarly, if the potential $u_0$ and $v_0$ are odd functions, we have
\begin{align*}
M^{\up}(x,y)&=\bar{K}^{\up}(-x,-y) \qquad M^{\dn}(x,y)=\bar{K}^{\dn}(-x,-y) \\
\bar{M}^{\up}(x,y)&=K^{\up}(-x,-y) \qquad \bar{M}^{\dn}(x,y)=K^{\dn}(-x,-y).
\end{align*}
Consequently, in these cases we only need to compute numerically one system, for instance system \eqref{nucleiKbarrati}. 
\end{remark}
\section{Initial Marchenko kernels, scattering matrix and Fourier transforms of reflection coefficients} \label{sect:Marchenko}
This section consists of two parts. In the first part we recall the Volterra integral equations that we solve to obtain the initial Marchenko kernels $\Omega_\ell(\alpha)$ and $\Omega_r(\alpha)$. 
In the second part we explain how compute the scattering matrix and the Fourier transforms  of the reflection coefficients $R$ and $L$.
\subsection{Initial Marchenko kernels}
\label{subsect:eqFredholm}
Following \cite[2.50a and 2.50b]{VanDerMee2013} we can say that, for $y \geq x \geq 0$, the Marchenko kernel $\Omega_\ell$ is connected to the auxiliary functions ${K}^{\dn}$ and $\bar{K}^{\dn}$ as follows:
\begin{equation}\label{Mar_K}
\zO_\ell(x+y)
+\displaystyle \int_x^\infty {K}^{\dn}(x,z)\, \zO_\ell(z+y) \, dz=-\bar{K}^{\dn}(x,y). 
\end{equation}
Similarly, for $y \leq x \leq 0$, the Marchenko kernel $\Omega_r$ is connected to the auxiliary functions ${M}^{\dn}$ and $\bar{M}^{\dn}$ in this way:
\begin{equation}\label{Mar_M}
\zO_r(x+y)+\displaystyle\int_{-\infty}^x {M}^{\up}(x,z) \,\zO_r(z+y) \, dz =-\bar{M}^{\up}(x,y).
\end{equation}
As a result, assuming known the auxiliary functions, \eqref{Mar_K} and \eqref{Mar_M} can be interpreted as structured Volterra integral equations having the initial Marchenko kernels $\Omega_\ell$ and $\Omega_r$ as their unknowns.

It is important to note that, from the computational point of view, each Marchenko kernel can be treated as a function of only one variable, as we only have to deal with the sum of the two variables. 
\subsection{The scattering matrix and the Fourier transforms of the reflection coefficients}
\label{subsect:scattering}
Let us begin by recalling that, as proposed in \cite{VanDerMee2013}, the coefficients of the scattering matrix $S(\lambda)$ can be represented as follows:
\begin{align}
T(\lambda)& = \frac{1}{a_{\ell 4}(\lambda)}= \frac{1}{a_{r1}(\lambda)},\label{T}\\
L(\lambda)& =\frac{a_{\ell2}(\lambda)}{a_{\ell4}(\lambda)}=-\frac{a_{r2}(\lambda)}{a_{r1}(\lambda)},\label{L}\\
R(\lambda)&=\frac{a_{r3}(\lambda)}{a_{r1}(\lambda)}=-\frac{a_{\ell 3}(\lambda)}{a_{\ell 4}(\lambda)} \label{R}
\end{align}
where the $\{a_{\ell j}(\lambda)\}$ and the $\{a_{rj}(\lambda)\}$ denote the entries of the transition matrices from the left and from the right, respectively. More precisely,
\begin{align}
\begin{cases}
a_{\ell1}(\lambda)&= 1-\displaystyle\int_{\R^+} e^{-i \lambda z} \, \bar{\Phi}^{\dn}(z) dz \\
a_{\ell2}(\lambda)&=-\displaystyle\int_{\R} e^{2i \lambda y} u_0(y) dy -\displaystyle \int_{\R} e^{i \lambda z} \Phi^{\dn}(z) dz \\
a_{\ell3}(\lambda)&=\displaystyle \int_{\R} e^{-2i \lambda y}   v_0(y) dy + \displaystyle \int_{\R} e^{-i \lambda z} \bar{\Phi}^{\up}(z) dz \\
a_{\ell4}(\lambda)&= 1+\displaystyle\int_{\R^+} e^{i \lambda z} {\Phi}^{\up}(z) dz,
\end{cases} 
\label{al}
\end{align}
where
\begin{align}
\bar{\Phi}^{\dn}(z)&= \displaystyle\int_{\R} u_0(y) \bar{K}^\dn(y,y+z) dy,\qquad 
\Phi^{\dn}(z)= \displaystyle \int_{-\infty}^{\frac z2 } u_0(y) {K}^{\dn}(y,z-y) dy,  \label{Phi1}\\
{\Phi}^{\up}(z)&=\displaystyle \int_{\R}   v_0(y) K^\up(y,y+z) dy, \qquad \bar{\Phi}^{\up}(z)=\displaystyle \int_{-\infty}^{\frac z2 }   v_0(y) \bar{K}^{\up}(y,z-y) dy,  \label{Phi2} 
\end{align}
and
\begin{align}
\begin{cases}
a_{r1}(\lambda)&= 1+\displaystyle\int_{\R^+} e^{i \lambda z} \Psi^{\dn}(z) dz \\
a_{r2}(\lambda)&=\displaystyle\int_{\R} e^{2i \lambda y} u_0(y) dy +\displaystyle \int_{\R} e^{i \lambda z}  \bar{\Psi}^{\dn}(z) dz  \\
a_{r3}(\lambda)&=-\displaystyle\int_{\R} e^{-2i \lambda y}   v_0(y) dy -\displaystyle \int_{\R} e^{-i \lambda z} \Psi^{\up}(z)  dz \\
a_{r4}(\lambda)&= 1-\displaystyle\int_{\R^+} e^{-i \lambda z} \bar{\Psi}^{\up}(z) dz  
\end{cases}
\label{ar}
\end{align}
\begin{align}
\Psi^{\dn}(z)&=\int_{\R} u_0(y) M^\dn(y,y-z) dy, \qquad 
\bar{\Psi}^{\dn}(z)=\int_{\frac{z}{2}}^{+\infty} u_0(y) \bar{M}^{\dn}(y,z-y) dy, \label{psi1} \\
\Psi^{\up}(z)&=\int_{\frac{z}{2}}^{+\infty}   v_0(y) M^{\up}(y,z-y) dy, \qquad \bar{\Psi}^{up}(z)=\int_{\R}   v_0(y) \bar{M}^\up(y,y-z) dy \label{psi2}. 
\end{align}

While the approximation of $T$ simply requires the computation of $a_{\ell4}(\lambda)$ and $a_{r1}(\lambda)$, that of $\rho$ and $\ell$ is more complicated. In fact, to approximate $\rho(\alpha)$ and $\ell(\alpha)$ we first have to compute the scattering coefficients by means of \eqref{al}-\eqref{ar}, then the reflection coefficients $R(\lambda)$ and $L(\lambda)$ by using \eqref{R} and \eqref{L} and, finally, $\rho(\alpha)$ and $\ell(\alpha)$ by resorting to the inverse and direct Fourier transforms as indicated in \eqref{rho} and \eqref{elle}. 

The stability of this numerical procedure essentially depends on the decay of $R(\lambda)$ and $L(\lambda)$ for $\lambda \to \pm \infty$ since the smoother the initial potential the faster their decay. If the initial potential has jump discontinuities then $R$ and $L$ decay as $\lambda^{-1}$ for $\lambda \to \infty$ while if $u_0 \in C^{\infty}(\mathbb{R})$ then $R$ and $L$ decay superpolynomially. 

Hence, this procedure is effective whenever the initial potential is smooth enough, that is at least $u_0 \in C(\mathbb{R})$. If this is not the case  the Fourier transforms $\rho(\alpha)$ and $\ell(\alpha)$ could be approximated by solving structured Fredholm integral equations stated in the following theorems. The development of an effective algorithm for solving these equations  is devoted to a subsequent paper.

\begin{theorem}
The function $\rho(\alpha)$ is the unique solution of each of  the following Fredholm integral equation of the second kind :
\begin{align}
\rho(\alpha)+\int_0^\infty \Phi^{\up}(z) \, \rho(z+\alpha) dz&=-\frac 12 v_0\left(\frac{\alpha}{2}\right)-\bar{\Phi}^{\up}(\alpha), \label{eq_rho1}\\
\rho(\alpha)+\int_0^\infty \Psi^{\dn}(z) \, \rho(z+\alpha) dz&=-\frac 12 v_0\left(\frac{\alpha}{2}\right)-\Psi^{\up}(\alpha), \label{eq_rho2}
\end{align}
where $\Phi^{\up}$ and $\bar{\Phi}^{\up}$ are given in \eqref{Phi1}-\eqref{Phi2} and $\Psi^{\dn}$ and $\Psi^{\up}$ are defined in \eqref{psi1}-\eqref{psi2}.
\end{theorem}
\begin{proof}
Let us first note that from \eqref{R}

\begin{equation}\label{al1_al3}
a_{\ell 4}(\lambda) R(\lambda)=-a_{\ell 3}(\lambda)
\end{equation}
where $a_{\ell 4}$ and $a_{\ell 3}$ are defined in \eqref{al}. Introducing the Heaviside function
$H(z)=1$ for $z \geq 0$ and $H(z)=0$ for $z<0$, taking into account that $$R(\lambda)=\ftran{\rho(\alpha)}$$ and using \eqref{al1_al3} we can write
$$\left(1+\ftran{\Phi^{\up}(-\alpha) H(-\alpha)}\right) \ftran{\rho(\alpha)} =-\ftran{\frac{1}{2} v_0\left(\frac{\alpha}{2}\right)+\bar{\Phi}^{\up}(\alpha)}.$$
Hence, applying the inverse Fourier transform and the convolution theorem, we have
$$\rho(\alpha)+\left( \Phi^{\up}(-\alpha) H(-\alpha) \right)* \rho(\alpha)=-\frac{1}{2} v_0\left(\frac{\alpha}{2} \right)-\bar{\Phi}^{\up}(\alpha),$$
and then the equation \eqref{eq_rho1} is an immediate consequence of the convolution definition and the Heaviside function. Equation \eqref{eq_rho2} can be obtained similarly, noting that $R(\lambda)$ satisfies the relation 
$$a_{r1}(\lambda) R(\lambda)=a_{r3}(\lambda)$$
and that
$$a_{r1}(\lambda)=1+ \ftran{\Psi^{\dn}(-\alpha) H(-\alpha)} \quad  \text{and} \quad  a_{r 3}(\lambda)=\ftran{\frac{1}{2} v_0\left(\frac{\alpha}{2} \right)+{\Psi}^{\up}(\alpha)}.$$
\end{proof}

We note that, from the numerical point of view, it is irrelevant if we solve \eqref{eq_rho1} rather than \eqref{eq_rho2}, since both are Fredholm integral equations of the second kind, equally structured. 

Applying the same technique we obtain the analogous
\begin{theorem}
The function $\ell(\alpha)$ is the unique solution of the two structured Fredholm integral equations of the second kind:
\begin{align}
\ell(\alpha)+\int_0^\infty {\Phi}^{\up}(z) \, \ell(z+\alpha) d z &=-\frac 12 u_0\left(\frac{\alpha}{2}\right)-\Phi^{\dn}(\alpha) \label{eq_elle1}\\
\ell(\alpha)+\int_0^\infty {\Psi}^{\dn}(z) \, \ell(z+\alpha) d z &=-\frac 12 u_0\left(\frac{\alpha}{2}\right)-\bar{\Psi}^{\dn}(\alpha),\label{eq_elle2}
\end{align}
where  $\Psi^{\dn}$ is defined in \eqref{psi1}  and $\Phi^{\up}$ and $\Phi^{\dn}$ are given in \eqref{Phi1}-\eqref{Phi2}.
\end{theorem}

We omit the proof, as it is analogous to the previous one, after noting that
$ L(\lambda)=\ftran{\ell(-\alpha)}.$

\section{The numerical method}
\label{sect:nummethod}
Let us now assume, for computational simplicity, that the support of the initial solution is bounded, that is 
\begin{equation}\label{assu0}
u_0(x)=0, \quad \text{for} \quad |x|>L, 
\end{equation}
which can be considered acceptable whenever $u_0(x) \to 0$ for $|x| \to \infty$, provided that $L$ is taken large enough. This hypothesis, as in part already proved in \cite{Fermo2014-PUB}, allows us to greatly simplify the algorithms for the computation of the auxiliary functions and also those for the computation of the Marchenko kernels and the Fourier transforms of the reflection coefficients. 

The method we propose provides successively the numerical solution  of:
\begin{enumerate}
\item the four systems \eqref{nucleiKbarrati}-\eqref{nucleiM} of Volterra integral equations for the computation of the four pairs of auxiliary functions;
\item the two Volterra integral equations \eqref{Mar_K}-\eqref{Mar_M} for the computation of the Marchenko kernels from the left and from the right $\Omega_\ell$ and $\Omega_r$, respectively;
\item the transition matrices from the left and from the right, the scattering matrix and then the inverse
 Fourier transforms $\rho$  of the reflection coefficients from the right $R$ and the Fourier transform $\ell$ of the reflection coefficient from the left $L$.
\end{enumerate}

Once  the Marchenko kernels $\Omega_\ell(\alpha)$ and $\Omega_r(\alpha)$ and the functions $\rho(\alpha)$ and $\ell(\alpha)$ have been obtained, the bound states $\{\lambda_j\}_{j=1}^n$ with their multiplicities $\{m_j\}_{j=1}^n$ and the norming constants $\{(\Gamma_\ell)_{js}, (\Gamma_r)_{js}\}$ are computed by applying to the monomial-exponential sums \eqref{S_ell}-\eqref{S_r} the matrix-pencil method proposed in \cite{Fermo2014-ETNA} and \cite{Fermo2014-AMC}.

\subsection{Auxiliary functions computation}
As said before, our numerical method for the solution of the Volterra systems \eqref{nucleiKbarrati}-\eqref{nucleiM} is greatly influenced by the hypothesis \eqref{assu0}. It implies a reduction of the auxiliary function supports, which allows us to develop algorithms that are simpler and numerically stable.

As proved in \cite{Fermo2014-PUB}, $\bar{K}^{\up}$ and $\bar{K}^{\dn}$ have the supports depicted in Figure \ref{support1}. Taking into account the symmetry properties \eqref{symmpropr1} or \eqref{symmpropr2}  of systems \eqref{nucleiKbarrati} and \eqref{nucleiK}, it is immediate to check that $supp(K^{\up})=supp(\bar{K}^{\dn})$ and $supp(K^{\dn})=supp(\bar{K}^{\up})$.

\begin{figure}[t]
\includegraphics[scale=0.36]{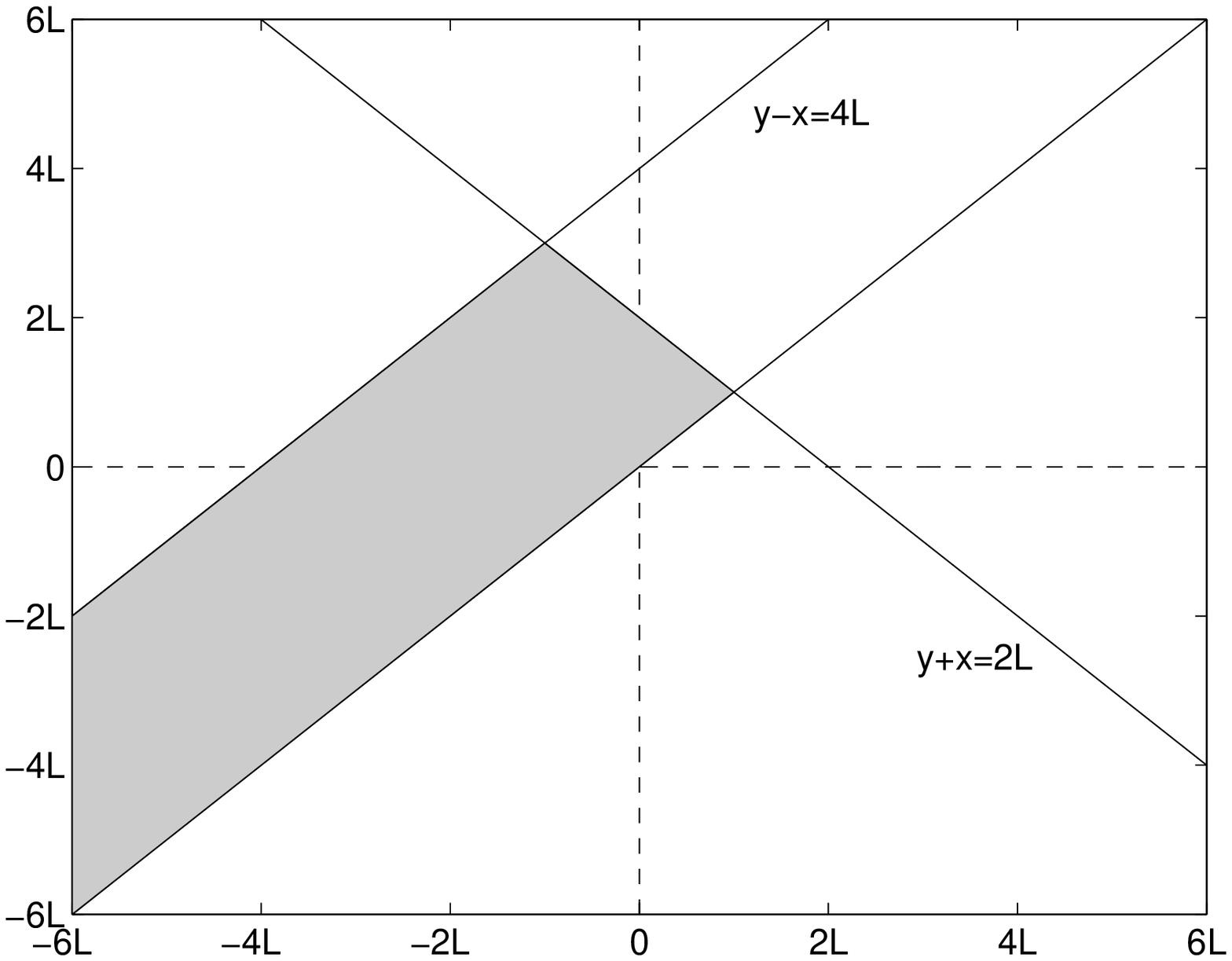}
\includegraphics[scale=0.36]{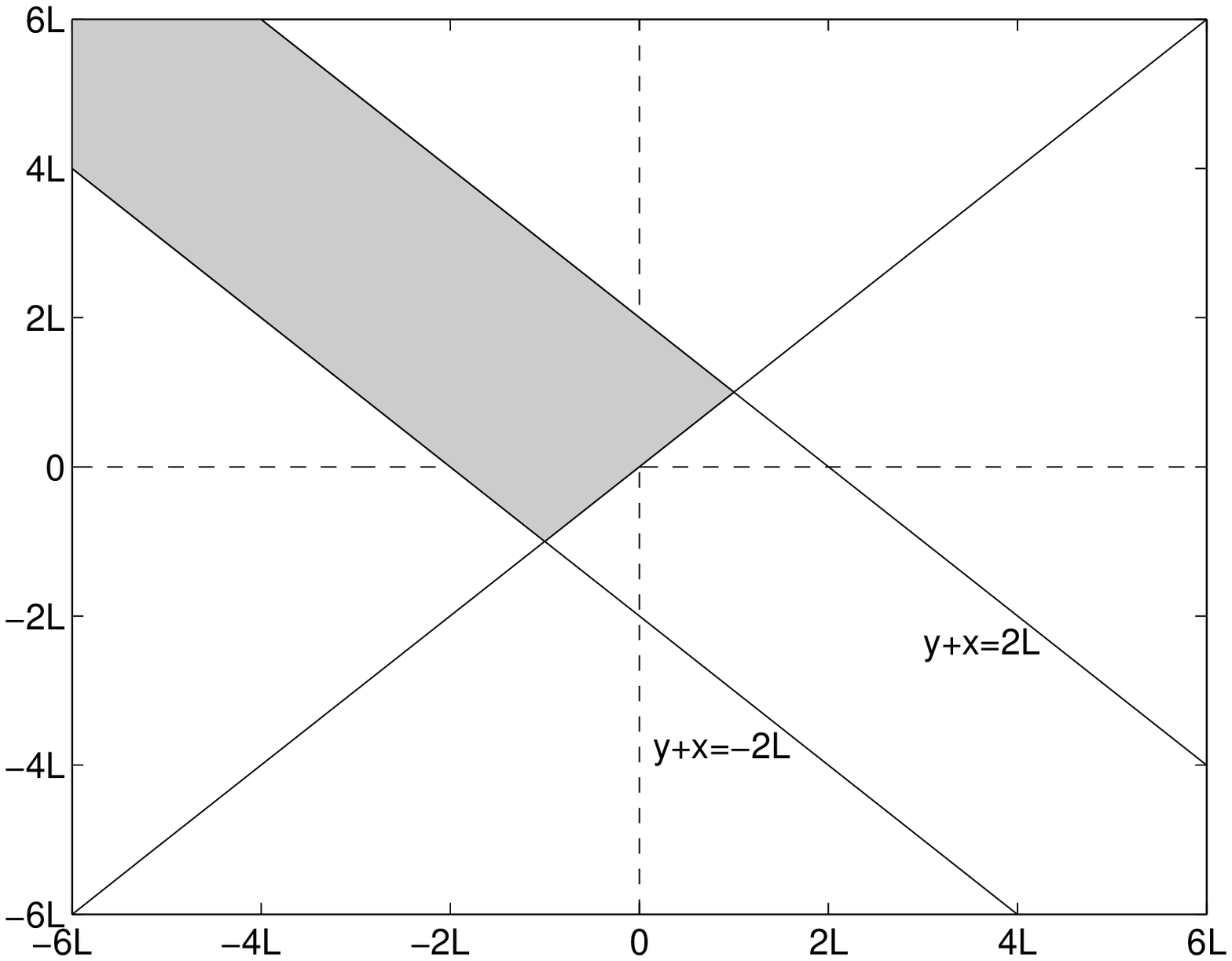}
\caption{\small Supports of the auxiliary functions $\bar{K}^{\up}$ and $K^{\dn}$ \label{support1} (to the left) and $\bar{K}^{\dn}$ and $K^{\up}$ (to the right)}
\end{figure}

For the numerical solution of system \eqref{nucleiKbarrati}, the following properties,  proved in \cite{Fermo2014-PUB}, are also important:
\begin{itemize}
\item[1.] If $x \leq -L$, whatever $h$, $\bar{K}^{\up}(x,y)$ and $K^{\dn}(x,y)$ are both constant on the line $y=x+h$.  For this reason we put $\bar{K}^{\up}(x,x+h)= {\mathcal{C}}^{\up}_{\bar{K},h}$ and ${K}^{\dn}(x,x+h)= {\mathcal{C}}^{\dn}_{K,h}$ for each given value $h$.
\item[2.] If $x<-L$ and $x+y>-2L$, $\bar{K}^{\dn}(x,y)$ and $K^{\up}$ are both  constant on each line $x+y=-2(L-h)$ for each $0<h<2L$. 
\end{itemize}
\begin{figure}[t]
\includegraphics[scale=0.36]{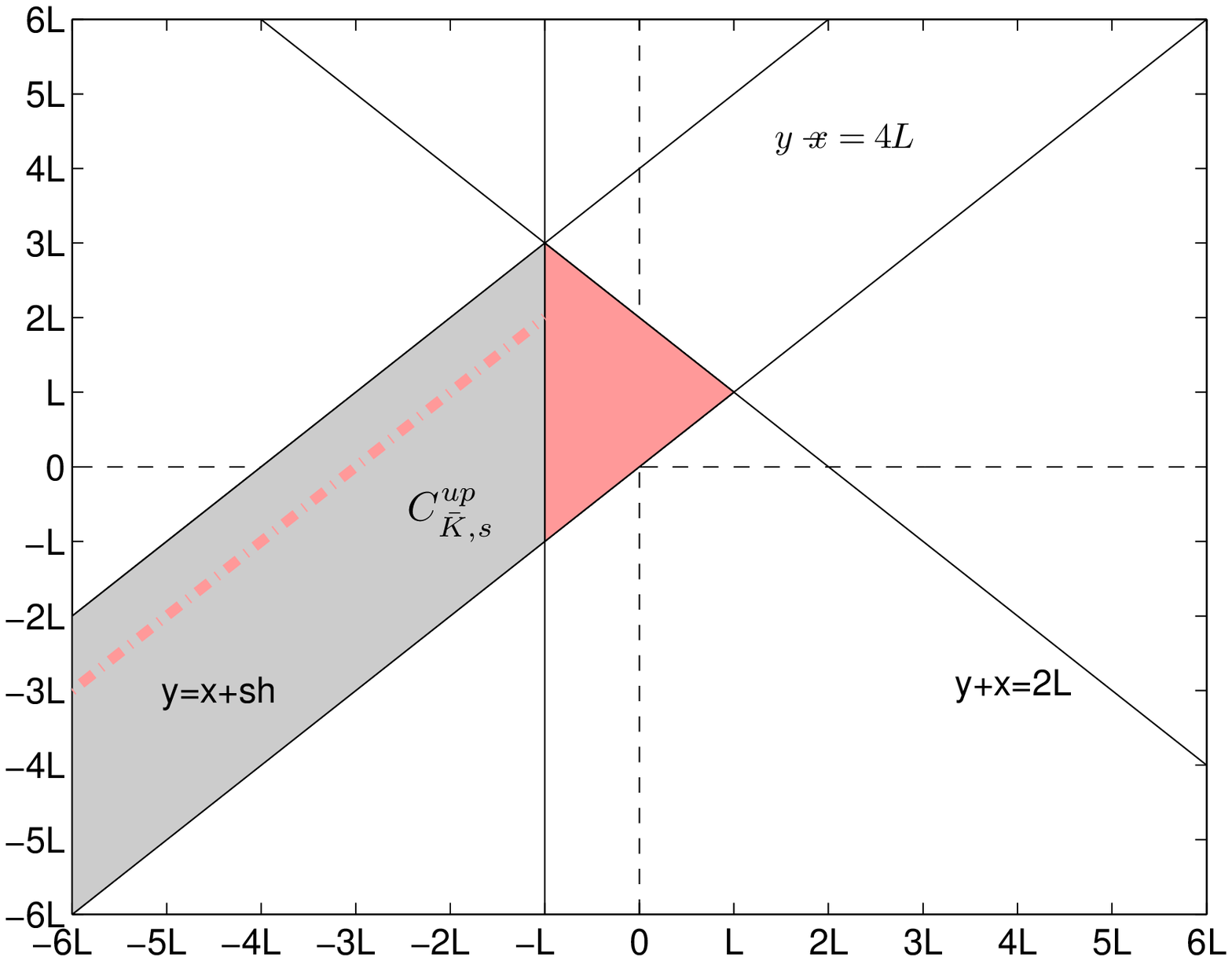}
\includegraphics[scale=0.36]{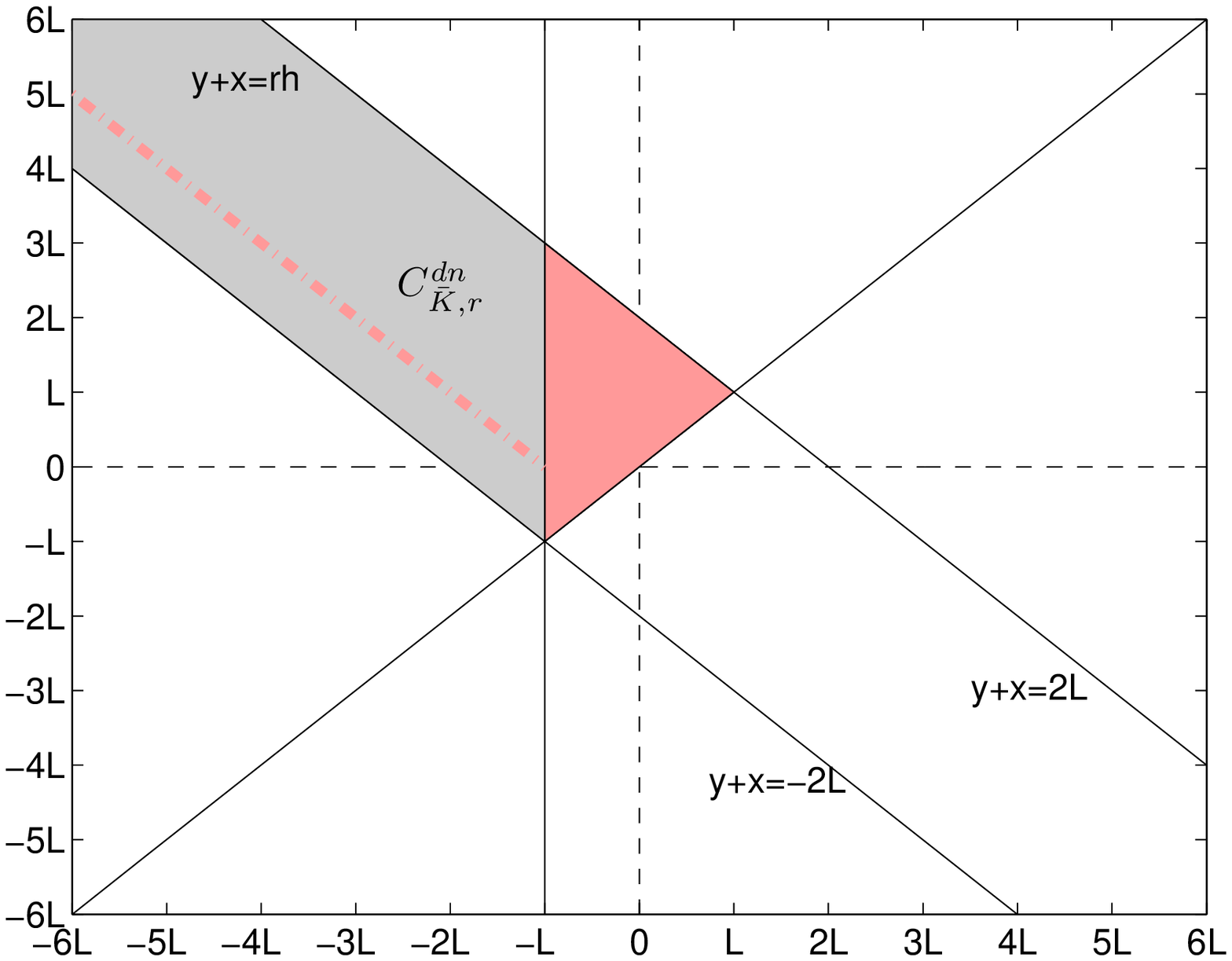}
\caption{Additional properties of $\bar{K}^{\up}$ and $K^{\dn}$ (to the left) and $\bar{K}^{\dn}$ and $K^{\up}$ (to the right) \label{property1}}
\end{figure}
These two results are graphically represented in Figure \ref{property1}, where $\bar{K}^{\dn}(x,x+h)=\mathcal{C}^{\dn}_{\bar{K},h}$ and $K^{\up}(x,x+h)=\mathcal{C}^{\up}_{K,h}$.

Analogous considerations, based on  results reported in \cite{Fermo2014-PUB} allow us to claim that the supports of $(\bar{M}^{\up}, \bar{M}^{\dn})$ are those depicted in Figure \ref{support2}. As for $(\bar{K}^{\up}, \bar{K}^{\dn})$ and $(K^{\up}, K^{\dn})$ as for the pairs $(\bar{M}^{\up}, \bar{M}^{\dn})$  we have additional properties very useful from the numerical point of view. With obvious meaning of the symbols, they are reported in Figure \ref{property2}.

\begin{figure}[t!]
\includegraphics[scale=0.36]{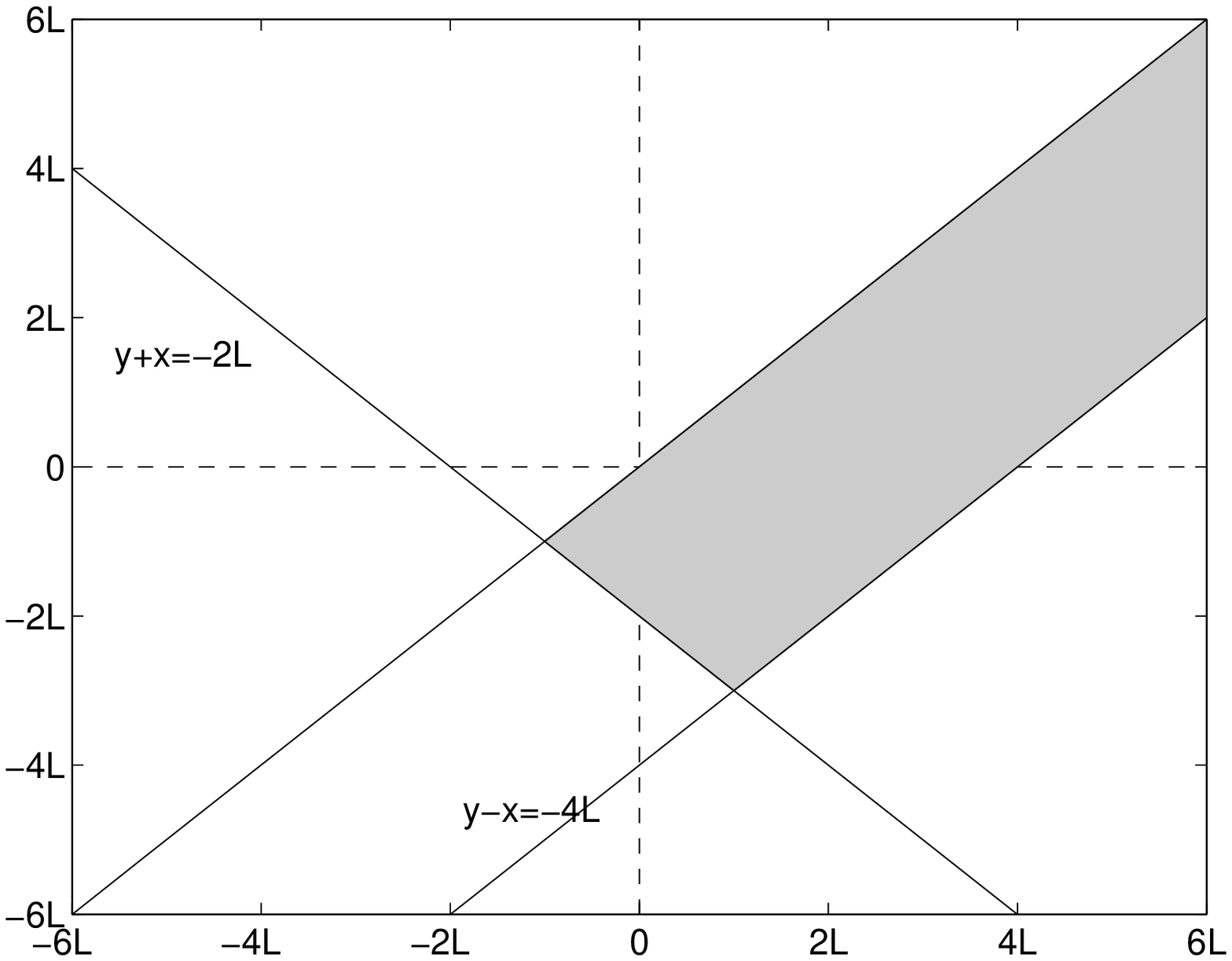}
\includegraphics[scale=0.36]{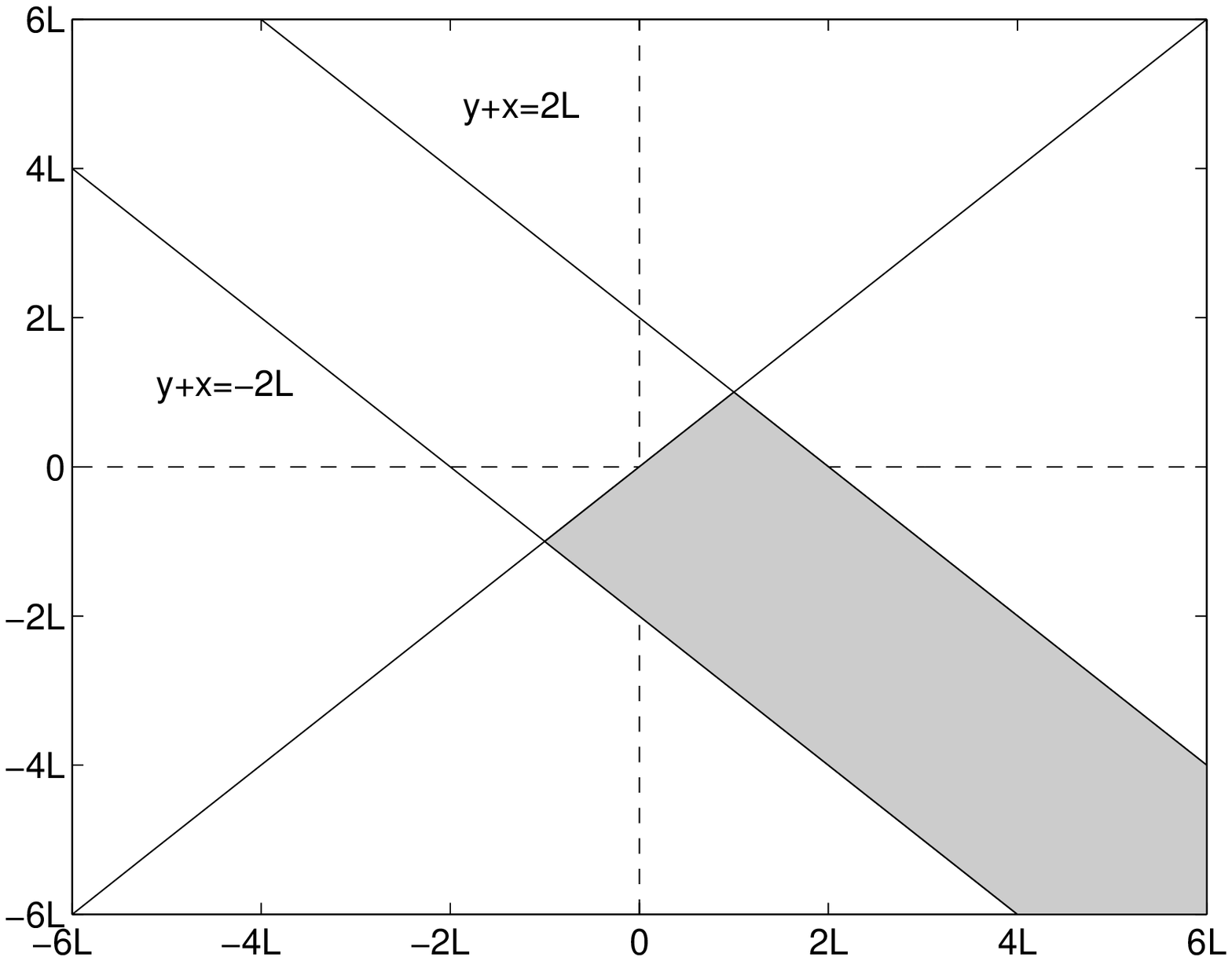}
\caption{\small Supports of the auxiliary functions ${M}^{\up}$ and $\bar{M}^{\dn}$ (to the left) and ${M}^{\dn}$ and $\bar{M}^{\up}$ (to the right) \label{support2}}
\end{figure}

\begin{figure}[t!]
\includegraphics[scale=0.36]{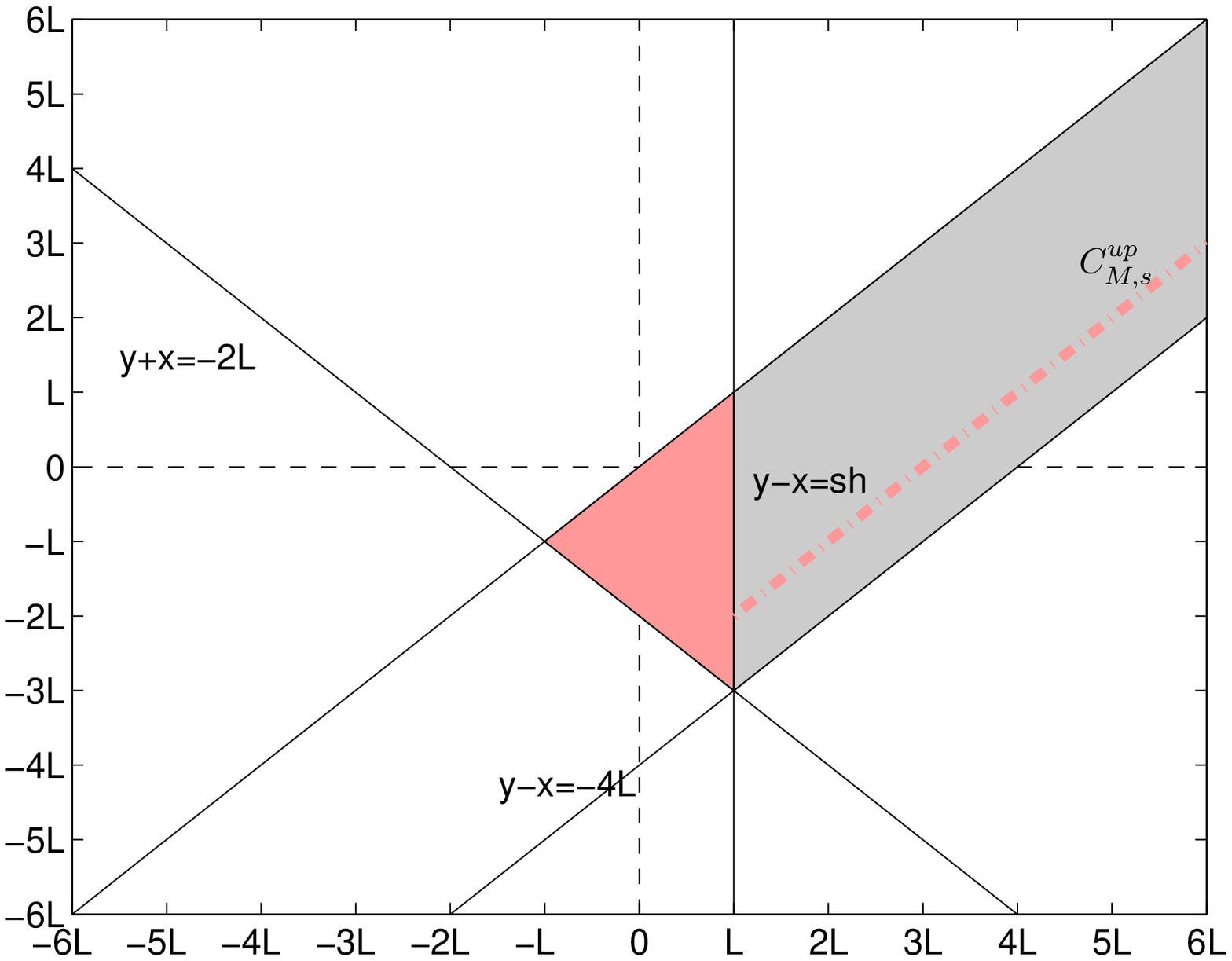}
\includegraphics[scale=0.36]{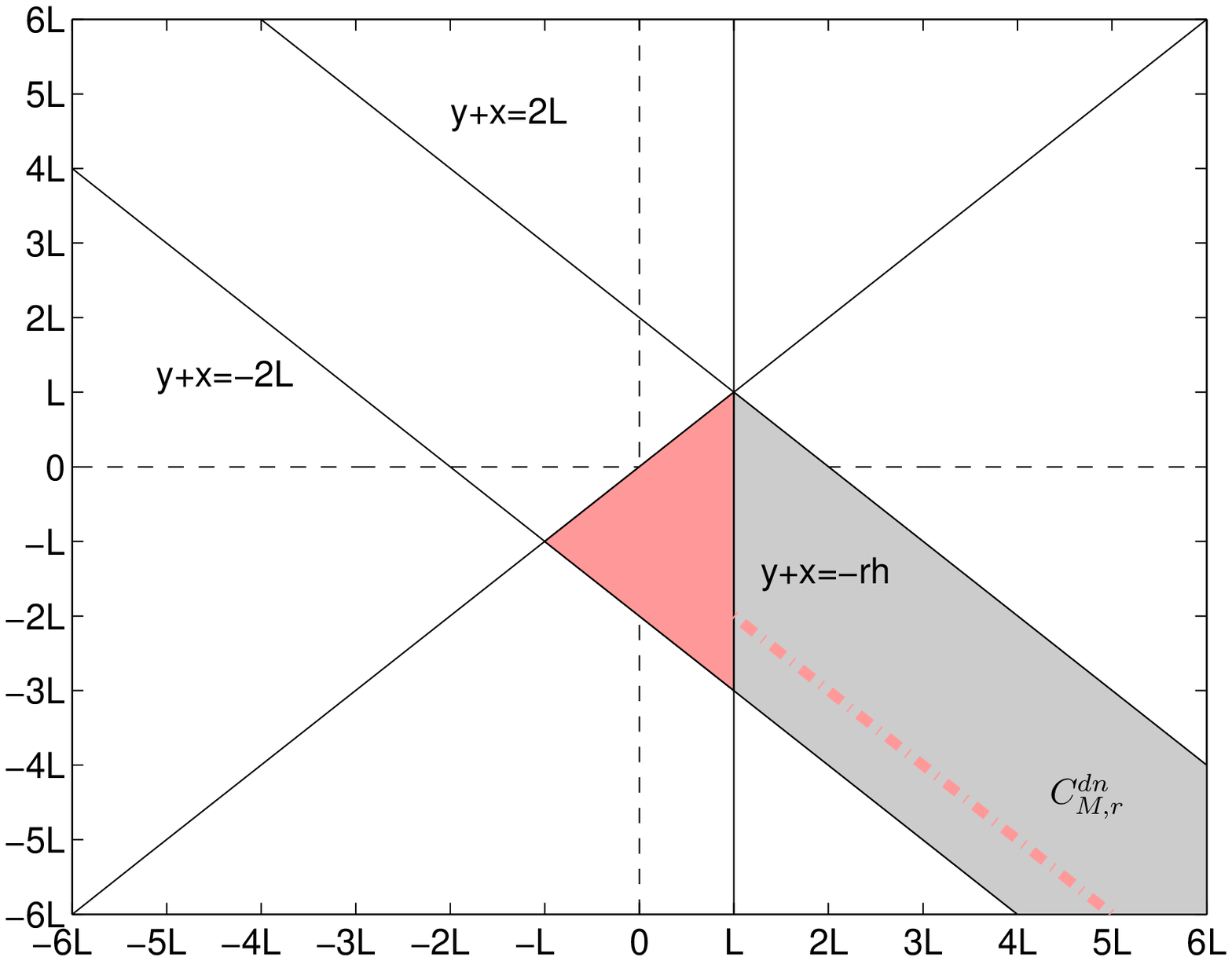}
\caption{Additional properties of ${M}^{\up}$ and $\bar{M}^{\dn}$ (to the left) and ${M}^{\dn}$ and $\bar{M}^{\up}$ (to the right)   \label{property2}}
\end{figure}

A simple inspection of Figures \ref{support1} and \ref{support2} makes it evident that the area where we need to compute $\bar{K}^{\up}$ and $\bar{K}^{\dn}$, as well as $K^{\up}$ and $K^{\dn}$,  is given by the orange triangle represented in  Figure \ref{square}. In the remaining areas of the respective supports their values are immediately obtained by using those of the orange triangle. The orange line shows, in particular, the values of the orange triangle we use to compute ($\bar{K}^{\up}, \bar{K}^{\dn}$) and ($K^{\up}, K^{\dn}$) in the point of the gray area. Similar considerations hold true for the computational area of the pairs ($\bar{M}^{\up}, \bar{M}^{\dn}$) and ($M^{\up}, M^{\dn}$) which is depicted in Figure \ref{square2}, with the analogous meaning of the symbols.

\begin{figure}[t]
\includegraphics[scale=0.36]{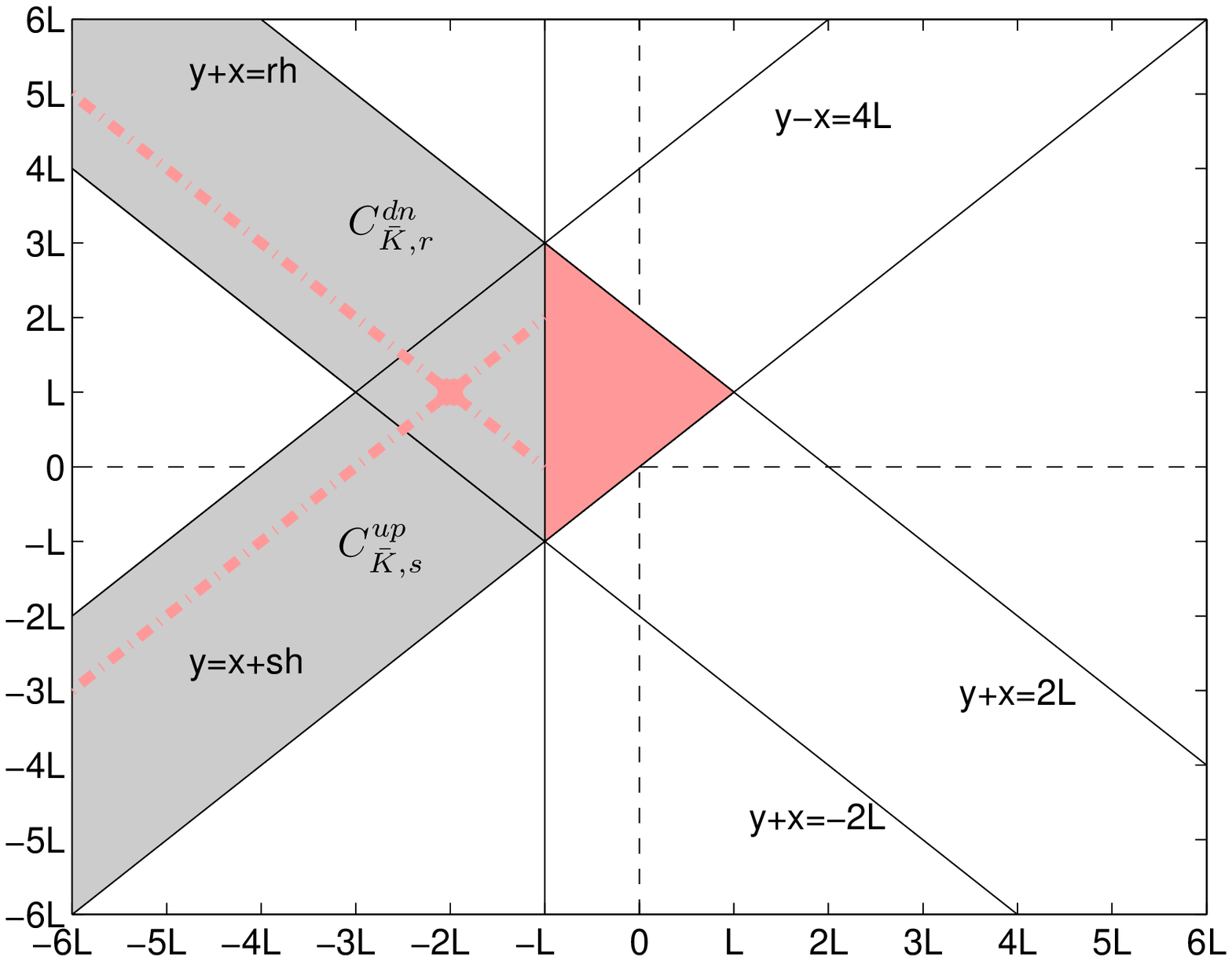}
\includegraphics[scale=0.36]{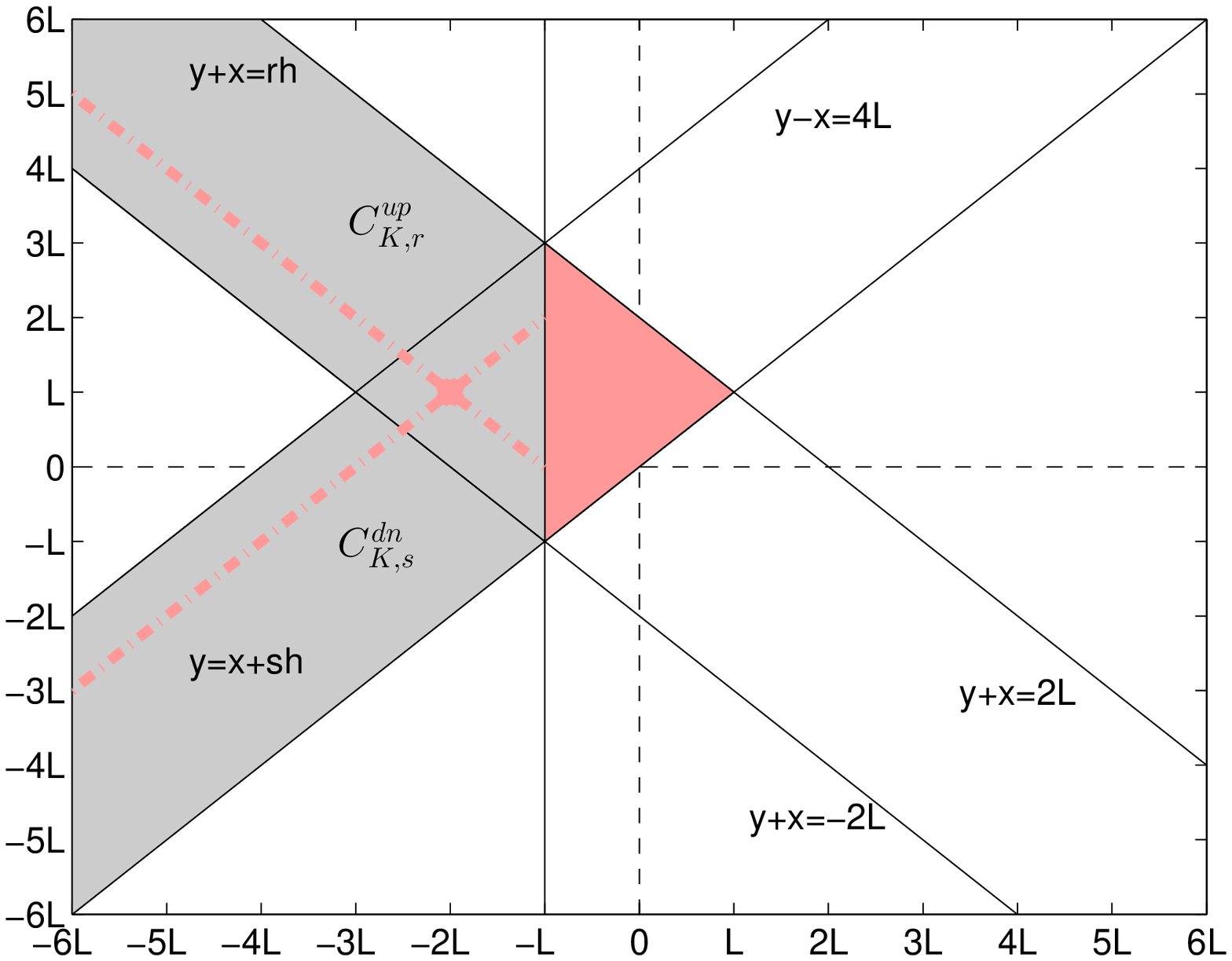}
\caption{Geometrical visualization of the computational area of $\bar{K}^{\up}$ and $\bar{K}^{\dn}$ (to the left) and ${K}^{\up}$ and ${K}^{\dn}$ (to the right) \label{square}}
\end{figure}

\begin{figure}[t!]
\includegraphics[scale=0.36]{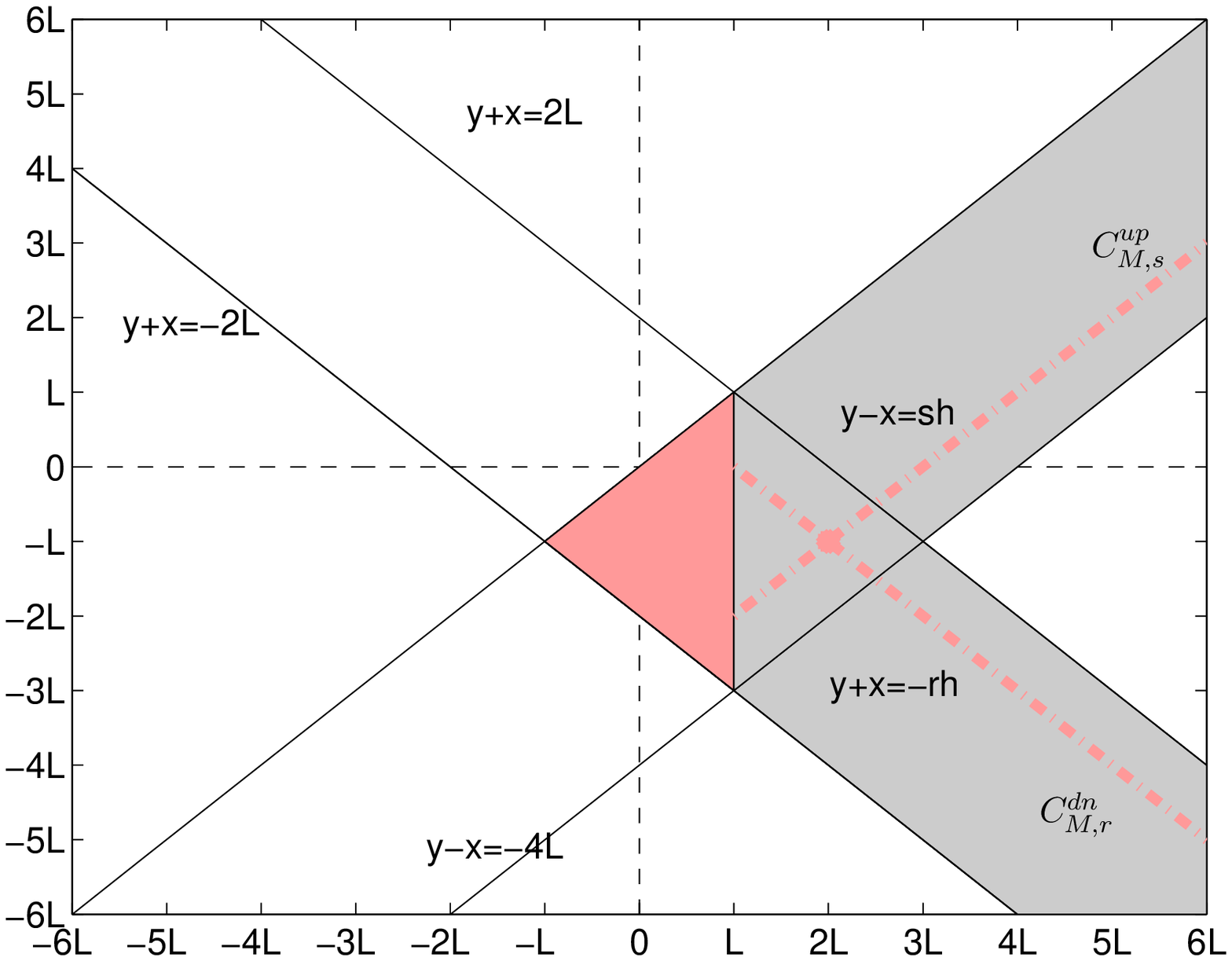}
\includegraphics[scale=0.36]{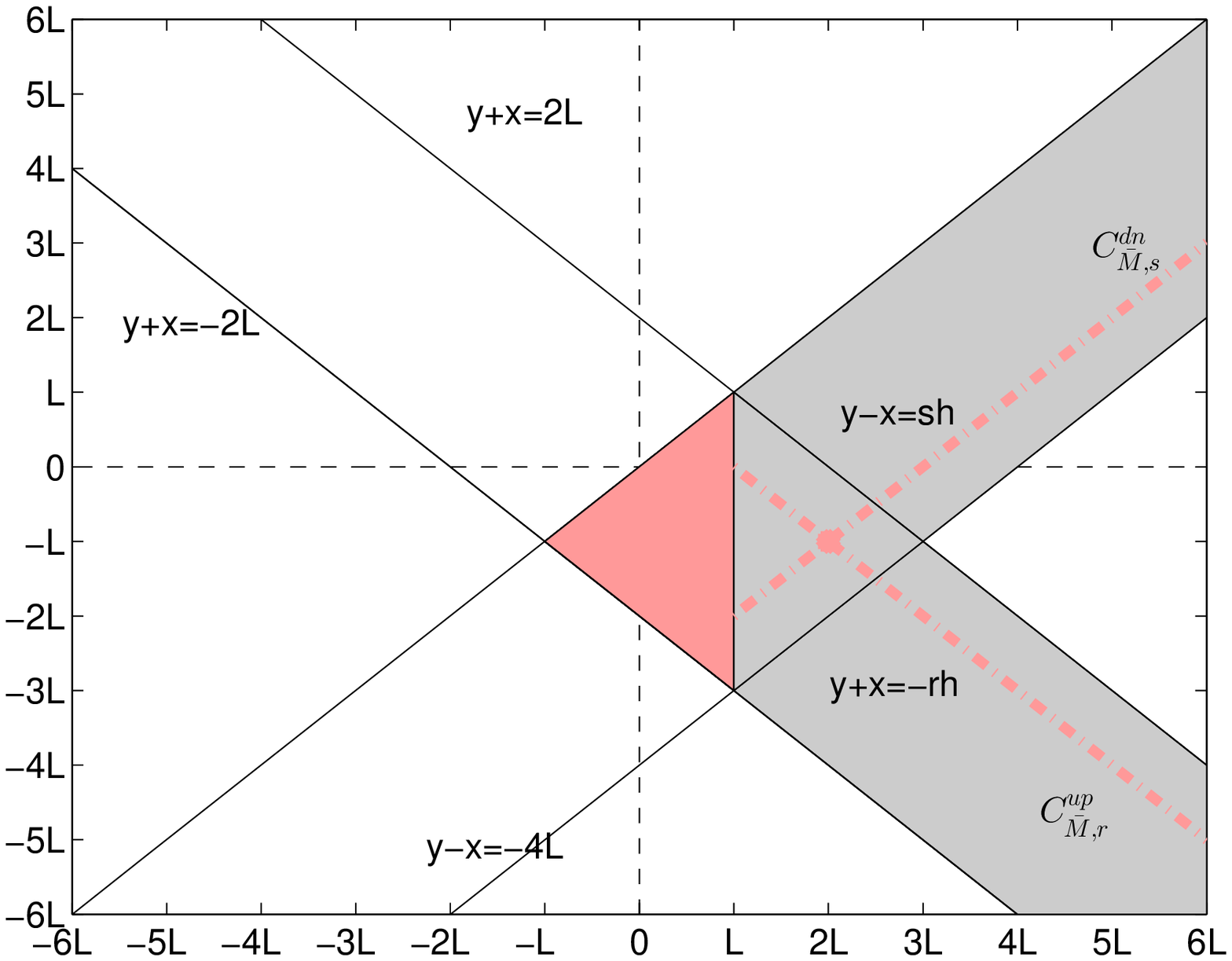}
\caption{Geometrical visualization of the computational area of $\bar{M}^{\up}$ and $\bar{M}^{\dn}$ (to the left) and ${M}^{\up}$ and ${M}^{\dn}$ (to the right)  \label{square2}}
\end{figure}

\begin{figure}[t!]
\includegraphics[scale=0.37]{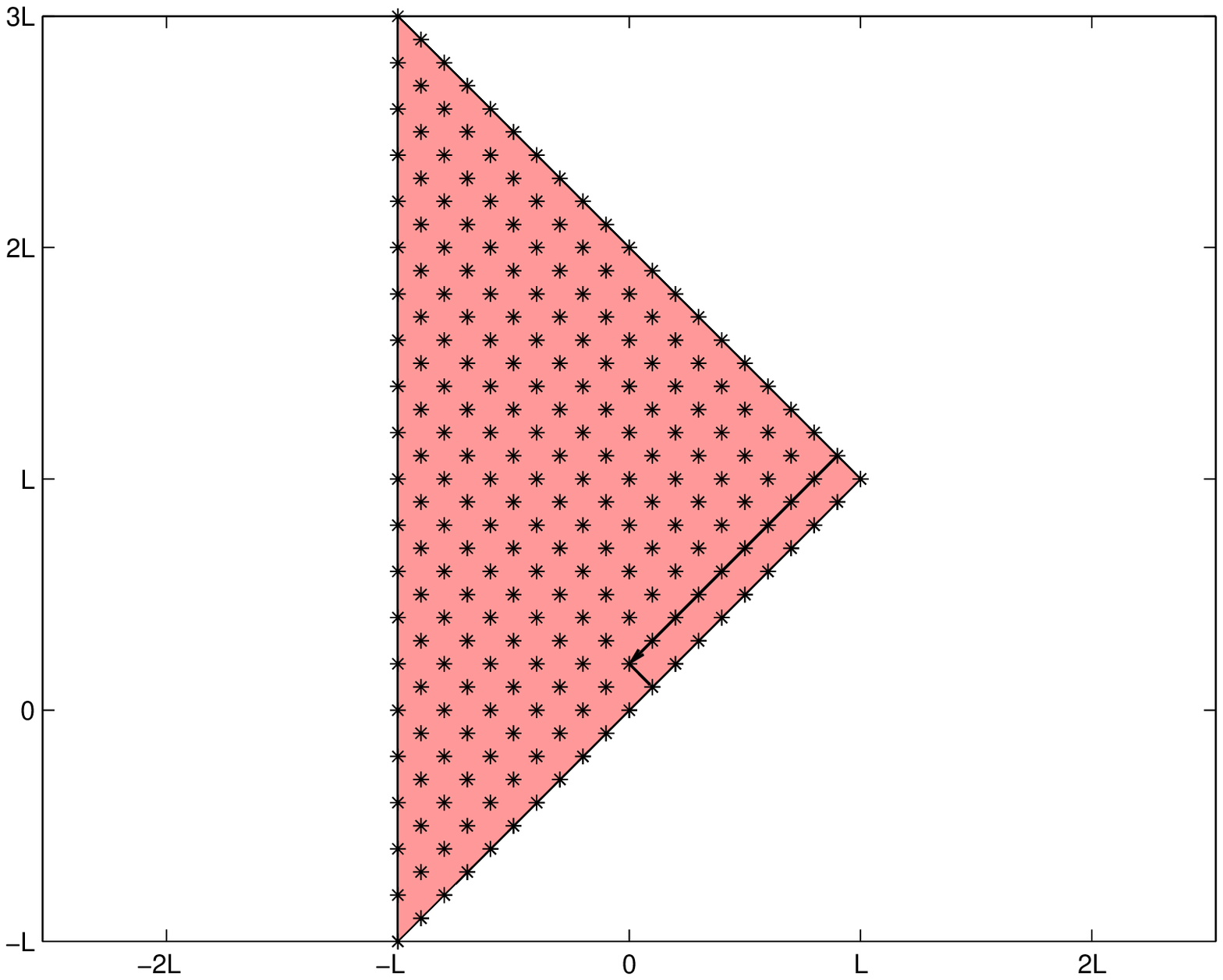} 
\includegraphics[scale=0.37]{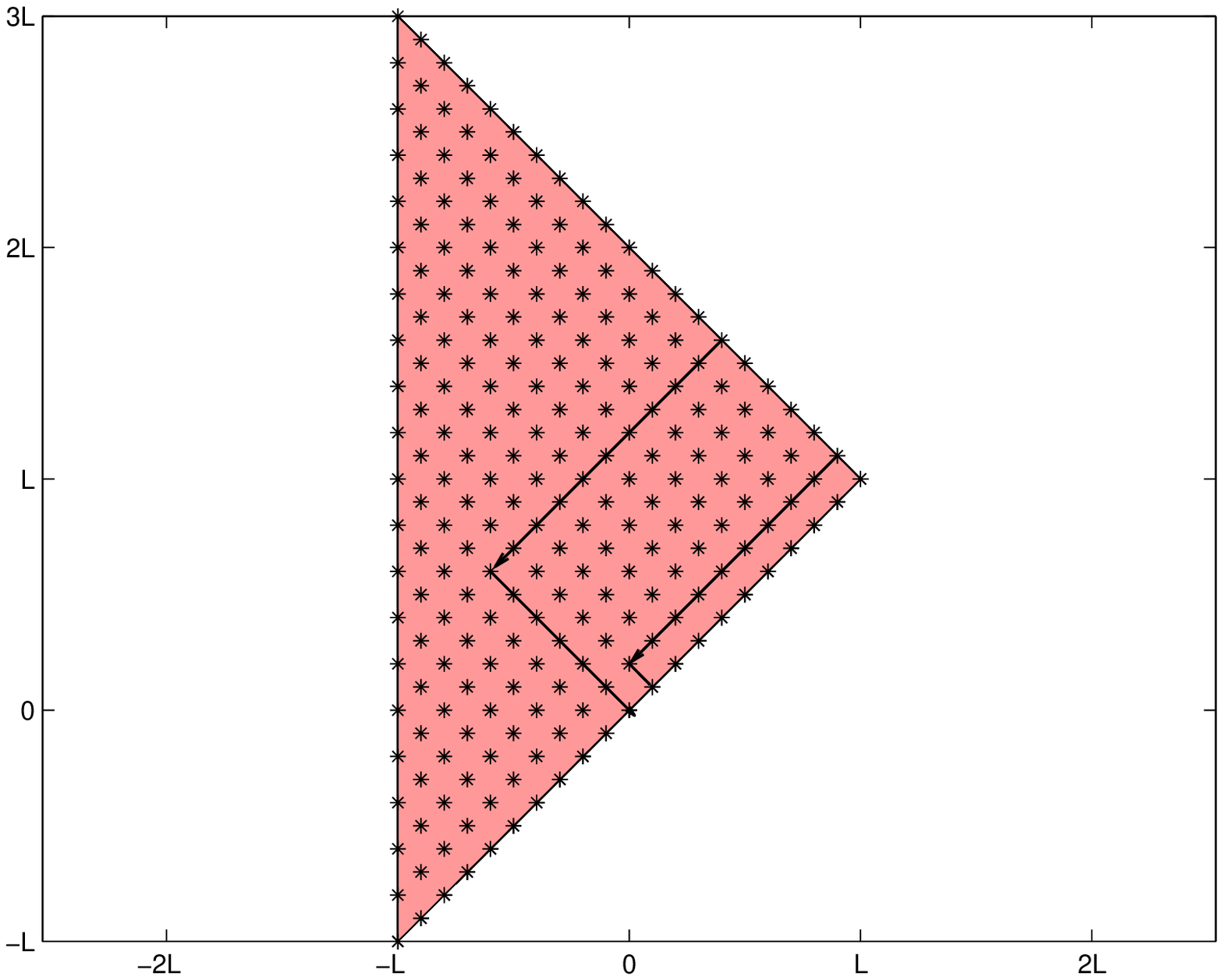} 
\caption{Sorting visualization of collocation points  in the triangle of $\bar{K}^{\up}$, $\bar{K}^{\dn}$, ${K}^{\up}$ and ${K}^{\dn}$  \label{collocazione}}
\end{figure}

\begin{figure}[t!]
\includegraphics[scale=0.37]{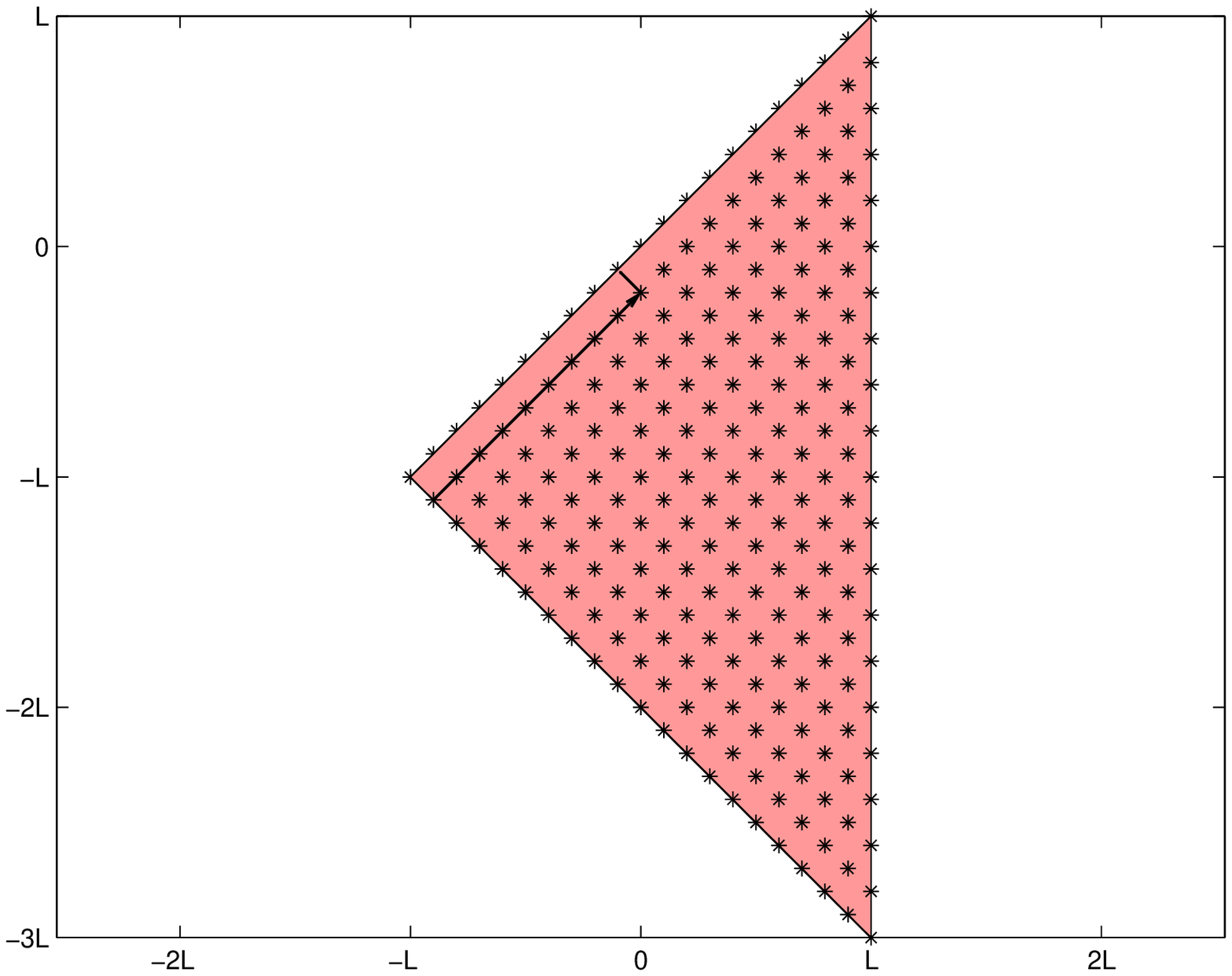}
\includegraphics[scale=0.37]{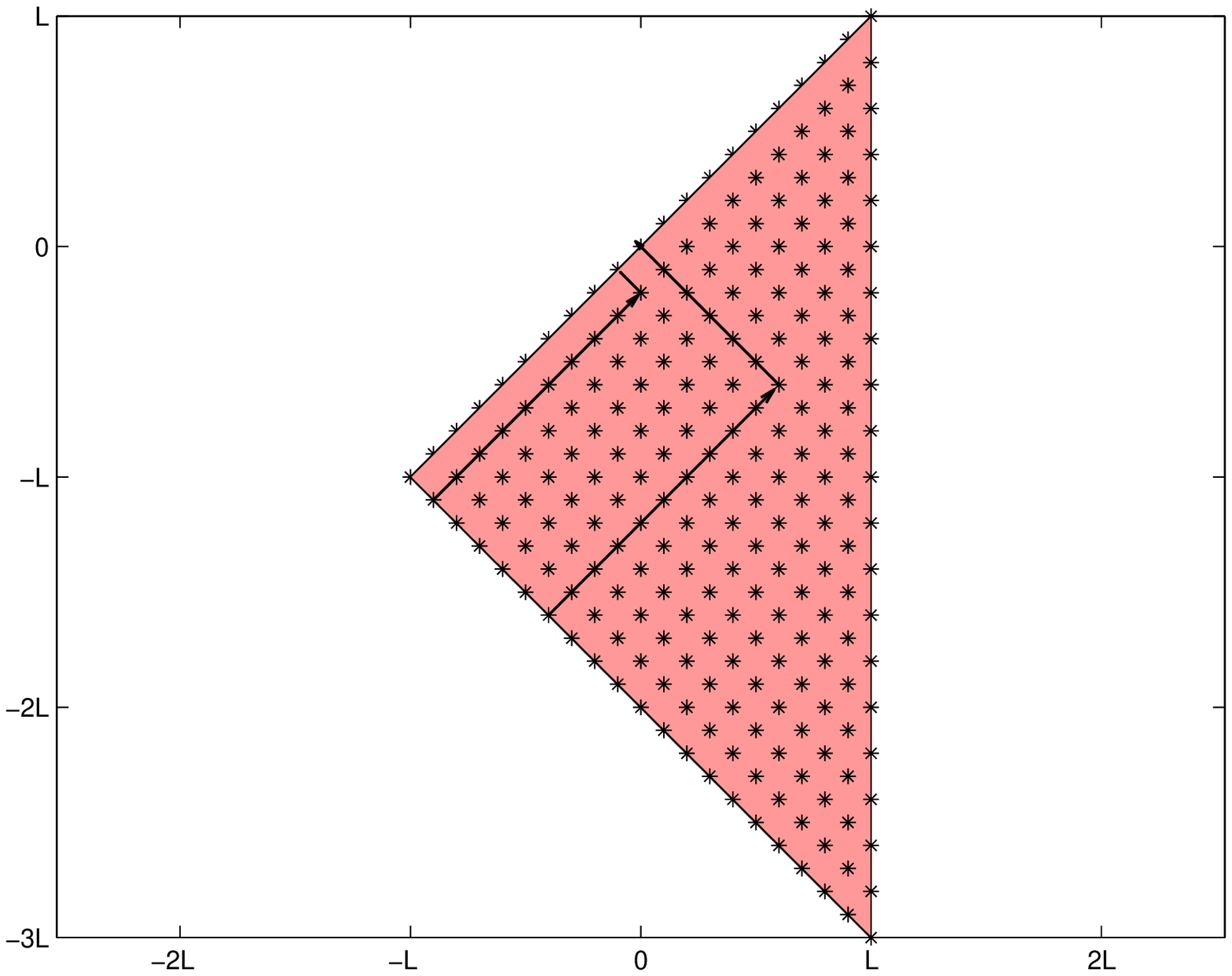} 
\caption{Sorting visualization of collocation points in the trangle of $\bar{M}^{\up}$, $\bar{M}^{\dn}$, ${M}^{\up}$ and ${M}^{\dn}$  \label{collocazioneM}}
\end{figure}

\subsection*{Algorithm} 
Given the initial solution $u_0$ and $v_0=u_0^*$ in the focusing case or $v_0=-u_0^*$ in the defocusing case, we have to solve  Volterra systems \eqref{nucleiKbarrati}-\eqref{nucleiM}. 

Let us start with the numerical solution of system \eqref{nucleiKbarrati}.
As noted before, under the hypothesis \eqref{assu0}, we can limit ourselves to solve this system in the triangular computational area represented in Figure \ref{square}, as the values of $\bar{K}^{\up}$ and $\bar{K}^{\dn}$ in the remaining parts of their support are then automatically known. 

The algorithm that we propose in this paper is more effective that the one reported in \cite{Fermo2014-PUB} whose aim was simply to check the effectiveness of our approach, highlighting the mathematical problems to overcome to obtain a satisfactory solution of the problem. 
Though the collocation strategy is the same used in \cite{Fermo2014-PUB}, the algorithm used here is more complex and effective.  In fact, it is based on the combined use of the trapezoidal rule, the composite Simpson quadrature formula and the  $3/8$ Simpson quadrature rule \cite[Section 3.1]{Stoer}, instead of only the composite trapezoidal quadrature formula used there. 

The first step is to fix a proper mesh in the computational area which can be done by fixing $n \in \mathbb{N}$, taking $h=\frac{L}{n}$ and introducing the following mesh points:
$$\mathcal{D}_k=\left \{ (x_i,x_{i+2k}),  \qquad x_i=ih, \quad i=n-k,n-k-1,...,-n+1,-n \right \}$$
where the index $k=0,\, \dots,\,2n$ identifies the line $y=x+2kh$ on which we want to compute the unknown functions, whereas $i$ labels the abscissa of the $i$-th mesh point on the line.

For the sake of simplicity, let us hereafter write $u$ and $v$ in place of $u_0$ and $v_0$, respectively. The computational strategy requires us to compute first  $\bar{K}^{\up}$ and $\bar{K}^{\dn}$ in the nodal points of the bisector $(x_i,x_i)$.
Consequently, recalling \eqref{propr1} and denoting by $\bar{K}^{\up}_{r,s}$, $\bar{K}^{\dn}_{r,s}$ the approximation of $\bar{K}^{\up}(x,y)$, $\bar{K}^{\dn}(x,y)$ in the nodal points of $\mathcal{D}_0$, we can write
\begin{align*}
\bar{K}^\up_{i,i} &= -\frac 12 \int_{{x}_{i}}^\infty u(z) v(z) \, dz=-\frac 12 \int_{{x}_{i}}^{x_{n+1}} u(z) v(z) \, dz \\ 
 \bar{K}^\dn_{i,i}&=\frac 12 v_{i}, \qquad i=n,n-1,...,\,-n+1,-n.
\end{align*} 
To approximate the above integral, it is convenient to use different quadrature formulae, according to the node $x_i$. More precisely for:
\begin{itemize}
\item[$\bullet$] $i=n$, being involved only two nodal points, we use the trapezoidal rule
$$\bar{K}^\up_{i,i} = -\frac{h}{4} \{ u_{n} v_n+ u_{n+1}v_{n+1}\}= -\frac{h}{4}  u_{n} v_n$$ as, for \eqref{propr1},  $u_{n+1}v_{n+1}=0$;
\item[$\bullet$]  $i=n-\ell$, $\ell=1,\,3,\,5,\,\dots,\,2n-1$, we apply the composite Simpson rule. Recalling that $u_{n+1}=v_{n+1}=0$, we then obtain
$$\bar{K}^\up_{i,i} = \frac{h}{3} \left[ u_{i}v_i+4 \sum_{j=1}^{\frac{\ell+1}{2}} u_{i+2j-1} v_{i+2j-1} + 2 \sum_{j=1}^{\frac{\ell+1}{2}-1} u_{i+2j} v_{i+2j} \right];$$
\item[$\bullet$] $i=n-\ell$, $\ell=2,\,4,\,6,\,\dots,\,2n$, noting that
\begin{align*}
\int_{{x}_{i}}^{x_{n+1}} u(z) v(z) \, dz=\left \{\int_{x_i}^{x_{i+3}}+ \int_{x_{i+3}}^{x_{n+1}} \right \} u(z)v(z) dz,
\end{align*} 
and that the first integral involves four nodes, while the second involves an odd number of nodes, we can apply the $3/8$ Simpson rule \cite[p.128]{Stoer}  for computing the first integral and  the composite Simpson quadrature formula for the second one. Hence, recalling again that $u_{n+1}=v_{n+1}=0$, we have
\begin{align*}
\bar{K}^\up_{i,i} & = \frac{3}{8}h \left[u_i v_i+ 3 u_{i+1} v_{i+1}+3 u_{i+2} v_{i+2} +u_{i+3}v_{i+3} \right] \\ & + \frac{h}{3} \left[ u_{i+3}v_{i+3}+4 \sum_{j=1}^{\frac{\ell}{2}} u_{i+2+2j} + 2 \sum_{j=1}^{\frac{\ell}{2}-1} u_{i+3+2j} v_{i+3+2j} \right].
\end{align*}
\end{itemize}

Once $\bar{K}^{\up}$ and $\bar{K}^{\dn}$ on the nodal points of the bisector $y=x$ are known, to evaluate them on the nodal points of the parallel lines to the bisector, we collocate system \eqref{nucleiKbarrati} on the nodes of the mesh $(x_i,x_{i+2k})$, taking  successively $k=1,\, \dots,\,2n$ and, fixing $k$, assuming  $i=n-k,\,\dots,\,-n+1,-n$. Hence, we can write
\begin{equation*}
\begin{cases}
\bar{K}^\up_{i,i+2k}+\displaystyle\int_{x_i}^\infty u(z) \, \bar{K}^\dn(z,z+2kh) \, dz=0, \\ \\
\bar{K}^\dn_{i,i+2k}-\displaystyle\int_{x_i}^{x_{i+k}} v(z) \, \bar{K}^\up(z,2(i+k)h-z) \, dz=\tfrac{1}{2} v_{i+k}.
\end{cases}
\end{equation*}
These formulae, taking into account the support of the  functions involved (Figure \ref{square}), reduce to
\begin{equation*}
\begin{cases}
\bar{K}^\up_{i,i+2k}+\displaystyle\int_{x_i}^{x_{n-k+1}} u(z) \, \bar{K}^\dn(z,z+2kh) \, dz=0, \\ \\
\bar{K}^\dn_{i,i+2k}-\displaystyle\int_{x_i}^{x_{i+k}} v(z) \, \bar{K}^\up(z,2(i+k)h-z) \, dz=\tfrac{1}{2} v_{i+k}.
\end{cases}
\end{equation*}

To compute the first integral 
$$I^1_{k,i}=\displaystyle \int_{x_i}^{x_{n-k+1}} u(z) \, \bar{K}^\dn(z,z+2kh) \, dz$$
we use different quadrature formulae, according to the node $x_i$. More precisely, fixing $k$, for: 
\begin{itemize}
\item[$\bullet$] $i=n-k$, being involved only two nodal points, we use the trapezoidal rule and then take 
$$I^1_{k,i} = \frac{h}{2} \{u_{n-k} \bar{K}^\dn_{n-k,n+k} + u_{n-k} \bar{K}^\dn_{n-k+1,n+k+1}\}= \frac{h}{2} u_{n-k} \bar{K}^\dn_{n-k,n+k},$$
as the nodal point $(x_{n-k+1}, x_{n+k+1})$ is outside of the support of $\bar{K}^\dn(x,y)$; 
\item[$\bullet$] $i=n-k-\ell$, with $\ell$ odd and $\ell \leq 2n-k$, applying the composite Simpson's rule, we obtain
\begin{align*}
I^1_{k,i} &= \frac{h}{3} \left[u_{i} \bar{K}^\dn_{i,i+2k}+ 4 \sum_{j=1}^{\frac{\ell+1}{2}} u_{i+2j-1} \bar{K}^{dn}_{i+2j-1,i+2j-1+2k} \right. \\ & \left. +2 \sum_{j=1}^{\frac{\ell+1}{2}} u_{i+2j-1} \bar{K}^{dn}_{i+2j,i+2j+2k}\right], 
\end{align*}
as  $\bar{K}^\dn_{n-k+1,n+k+1}=0$.
\item[$\bullet$] $i=n-k-\ell$, with $\ell$ even and $\ell \leq 2n-k$, noting that
\begin{align*}
I^1_{k,i}= \displaystyle \left \{\int_{x_{i}}^{x_{i+3}} +\displaystyle\int_{x_{i+3}}^{x_{n-k+1}} \right\} u(z) \, \bar{K}^{\dn}(z,z+2kh) \, dz,
\end{align*}
we apply the $3/8$ Simpson's rule for the first integral and the composite Simpson's quadrature formula for the second one. Hence, we have
\begin{align*}
I^1_{k,i} & = \frac{3}{8}h \left[ u_{i}\bar{K}^{\dn}_{i,i+2k}+ 3 u_{i+1}\bar{K}^{\dn}_{i+1,i+1+2k}+3 u_{i+2}\bar{K}^{\dn}_{i+2,i+2+2k} + u_{i+3}\bar{K}^{\dn}_{i+3,i+3+2k}  \right] \\ & +
\frac{h}{3} \left[u_{i+3}\bar{K}^{\dn}_{i+3,i+3+2k}+  4 \sum_{j=1}^{\frac{\ell}{2}} u_{i+2+2j} \bar{K}^{\dn}_{i+2+2j,i+2+2j+2k} \right. \\ & \hspace*{4cm}\left. +2 \sum_{j=1}^{\frac{\ell}{2}-1} u_{i+3+2j} \bar{K}^{\dn}_{i+3+2j,i+3+2j+2k}\right]
\end{align*}
as the nodal point $(x_{n-k+1},x_{n+k+1})$ is outside the support of $\bar{K}^{\dn}(x,y)$.
\end{itemize}

The computation of the second integral
$$I^2_{k,i}=\displaystyle\int_{x_i}^{x_{i+k}} v(z) \, \bar{K}^\up(z,2(i+k)h-z) \, dz, $$
is also based on the use of quadrature formulae, essentially dependent on the line $y=x+2kh$. More precisely, for:
\begin{itemize}
\item[$\bullet$] $k=1$, as only two nodal points are involved, we apply the trapezoidal rule, obtaining
$$I^2_{k,i} = \frac{h}{2} \{v_{i} \bar{K}^\up_{i,i+2} + v_{i+1} \bar{K}^\up_{i+1,i+1} \};$$
\item[$\bullet$] $k=2,4,6,...,\,2n$, we use the composite Simpson quadrature formula. Proceeding in this way we obtain for $i=n-k,...,\,-n$
\begin{align*}
I^2_{k,i} & = \frac{h}{3} \left[v_{i} \bar{K}^\up_{i,i+2k}+ 4 \sum_{j=1}^{\frac{\ell}{2}} v_{i+2j-1} \bar{K}^{\up}_{i+2j-1,i-2j+1+2k} \right. \\ & \left. +2 \sum_{j=1}^{\frac{\ell}{2}-1} v_{i+2j} \bar{K}^{\up}_{i+2j,i-2j+2k}+ v_{i+k} \bar{K}^\up_{i+k,i+k} \right] \\
& = \frac{h}{3} v_{i} \bar{K}^\up_{i,i+2k}+ w_{k,i},
\end{align*}
where $w_{k,i}$ is the sum of the $\bar{K}^{\up}$ values in the nodal points belonging to the bisector and the previous parallels. In fact, the $\bar{K}^{\up}$ values of the first sum belong to the lines $y=x+[2k-2(2j-1)]h$, those of the second one belong to the lines   $y=x+[2k-4j]h$ and the last term to $y=x$.
\item[$\bullet$] $k=3,5,7,...,\,2n-1$, we write
\begin{align*}
I^2_{k,i}= \displaystyle \left \{\int_{x_{i}}^{x_{i+3}} +\displaystyle\int_{x_{i+3}}^{x_{i+k}} \right\} v(z) \, \bar{K}^\up(z,2(i+k)h-z) \, dz 
\end{align*} 
and then we use the $3/8$ Simpson rule for the first integral and again the composite Simpson quadrature formula for the second one: 
\begin{align*}
I^2_{k,i} &= \frac{3}{8}h \left[ v_{i} \bar{K}^{\up}_{i,i+2k}+ 3 v_{i+1}\bar{K}^{\up}_{i+1,i+2k-1}+3 v_{i+2} \bar{K}^{\up}_{i+2,i+2k-2} + v_{i+3}\bar{K}^{\up}_{i+3,i+2k-3}  \right] \\ & +
\frac{h}{3} \left[v_{i+3}\bar{K}^{\up}_{i+3,i+2k-3}+  4 \sum_{j=1}^{\frac{\ell}{2}} v_{i+2+2j} \bar{K}^{\up}_{i+2+2j,i-2-2j+2k} \right. \\ & \left. +2 \sum_{j=1}^{\frac{\ell}{2}-1} u_{i+3+2j} \bar{K}^{\up}_{i+3+2j,i-3-2j+2k}+ v_{i+k}\bar{K}^{\up}_{i+k,i+k}\right] \\ &= \frac{3}{8}h v_{i} \bar{K}^{\up}_{i,i+2k}+ w_{k,i},
\end{align*}
where $w_{k,i}$ is known, being a linear combination of $\bar{K}^{\up}$ values already computed.

Once the integrals have been approximated as described above, we obtain the $2n$ following structured  systems of order $2(2n+1-k)$, $k=1,\, \dots,\,2n$ 
\begin{equation}\label{system}
\begin{cases}
\mathbf{\bar{k}}^{\up}_k + \mathbf{U}_{k,1} \mathbf{\bar{k}}^{\dn}_k=\mathbf{0} \\
\mathbf{U}_{k,2} \mathbf{\bar{k}}^{\up}_k+ \mathbf{\bar{k}}^{\dn}_k= \mathbf{v}_{k}-\mathbf{w}_{k}
\end{cases}
\end{equation}
that  allow us to compute the functions $\bar{K}^{\up}$ and $\bar{K}^{\dn}$ in the $2n+1-k$ nodal points of $\mathcal{D}_k$ as
\begin{align*}
\mathbf{\bar{k}}^{\up}_k &=(\bar{K}^{\up}_{n-k,n+k},\,\bar{K}^{\up}_{n-k-1,n+k-1},\,\dots,\,\bar{K}^{\up}_{-n+1,-n+2k+1},\,\bar{K}^{\up}_{-n,-n+2k})^T \\
\mathbf{\bar{k}}^{\dn}_k &=(\bar{K}^{\dn}_{n-k,n+k},\,\bar{K}^{\dn}_{n-k-1,n+k-1},\,\dots,\, \bar{K}^{up}_{-n+1-n+2k+1},\,\bar{K}^{\dn}_{-n,-n+2k})^T. 
\end{align*}
Notice that   $\mathbf{U}_{k,1},\,\mathbf{U}_{k,2}$ are the following structured matrices:
\begin{align*}
\mathbf{U}_{k,2}&=-c_k \, h  \, diag(v_{n-k},\,v_{n-k-1},\,\dots,\,v_{-n+1},\,v_{-n}) 
\end{align*}
with $c_1=1/2$, $c_2=c_4=...=c_{2n}=1/3$ and $c_3=c_5=...=c_{2n-1}=3/8$ and
\begin{align*}
\mathbf{U}_{k,1}&= h \left( \begin{matrix}
1/2 & & & & &  \\ 
4/3 & 1/3 & & & & \\ 
9/8 & 9/8 & 3/8 & & & \\ 
4/3 & 2/3 & 4/3 & 1/3 & & \\ 
4/3 & 17/24 & 9/8 & 9/8 & 3/8  \\ 
\vdots & & & \vdots & &  \ddots \\
\vdots & & & \vdots & & & \ddots \\
\vdots & & & \vdots & & & & \ddots \\
4/3 & 2/3 & \hdots & \hdots & & 17/24 & 9/8 & 9/8 & 3/8  \\ 
4/3 & 2/3 & 4/3 & \hdots & \hdots  &  & 4/3 & 2/3 & 4/3 & 1/3 \\
\end{matrix} \right) \times
\\ & \hspace*{1cm} diag(u_{n-k},\,u_{n-k-1},\,\dots,\,u_{-n+1},\,u_{-n}).
\end{align*}
The most obvious computational strategy is to reduce \eqref{system} to a sequence of $n-k$ systems of order two. However, our numerical experiments indicate that the numerical stability increases by using a suitable iterative method. 

It requires solving iteratively the system 
\begin{equation}\label{system1}
(\mathbf{I}-\mathbf{U}_{k,1} \mathbf{U}_{k,2}) \, \mathbf{\bar{K}}^{\up}_k= \mathbf{U}_{k,1} \, \mathbf{w}_{k},
\end{equation}
and then computing    
\begin{equation}\label{system2}
\mathbf{\bar{K}}^{\dn}_k= \mathbf{V}_{k}- \mathbf{U}_{k,2} \,\mathbf{\bar{K}}^{\up}_k.
\end{equation}
The matrix of system \eqref{system1}, for $h$ small enough, is diagonally dominant as each nonzero element of $\mathbf{U}_{k,1} \mathbf{U}_{k,2}$ contains a factor $h^2$, so that the Gauss-Seidel method is a suitable choice of iteration method, assuming as an initial vector  the values of $\mathbf{\bar{k}}^{\up}_k$ in the previous parallel, that is taking in the $(k+1)$th parallel to the bisector
\begin{equation}
(\mathbf{\bar{k}}^{\up}_{k+1})^{(0)} = \mathbf{\bar{k}}^{\up}_{k} \qquad k=0,1,\,\dots,\,2n-1.
\end{equation}
As $\mathbf{I}-\mathbf{U}_{k,1} \mathbf{U}_{k,2}$ is lower triangular, it is of course possible to solve it by a descending technique.
\end{itemize}   

\begin{remark}
Once we have solved system \eqref{nucleiKbarrati} we can immediately deduce the solution of system \eqref{nucleiK} taking into account Remark \ref{remark1}. In any case, we note that,  as the computational area of system \eqref{nucleiKbarrati} is the same as that of \eqref{nucleiK}, 
the algorithm to solve \eqref{nucleiK} is analogous to that adopted for system \eqref{nucleiKbarrati}.

The same comparative considerations hold true for the computation of ($\bar{M}^{\up}$, $\bar{M}^{\dn}$) and (${M}^{\up}$, ${M}^{\dn}$)  in the nodal points of their computational area. Moreover, although the computational area for $(\bar{M}^{\up}$, $\bar{M}^{\dn})$ is not the same as that for $(\bar{K}^{\up}$, $\bar{K}^{\dn})$,  the technique for their computation is essentially the same.

Noting that (Figures \ref{square}, \ref{square2}) the two computational areas  are symmetric with respect to each other, we first have to compute ($\bar{M}^{\up}$, $\bar{M}^{\dn}$) in the bisector and then on the parallel lines $y=x-2kh$, $k=1,2,...,2n$. Furthermore, to compute $\bar{M}^{\dn}$   in the bisector we can adopt the same algorithm for $\bar{K}^{\up}$ as relations \eqref{propr1} and \eqref{propr4} indicate. A comparison between the systems \eqref{nucleiKbarrati} and \eqref{nucleiMbarrati} also suggests to approximate the first integral in \eqref{nucleiMbarrati} by a simple adaptation of the method developed for the second one in \eqref{nucleiKbarrati}, as well as the second integral of \eqref{nucleiMbarrati} by adapting the method for the first integral of \eqref{nucleiKbarrati}.
\end{remark}

\subsection{Marchenko kernel computation}
To compute $\Omega_\ell$ and $\Omega_r$, that is to solve the integral equations \eqref{Mar_K} and \eqref{Mar_M}, we first note that \eqref{assu0} implies the boundedness of their supports. In fact, as proved in \cite[Lemma 5.1]{Fermo2014-PUB}, \eqref{assu0} implies that
$$supp(\Omega_\ell)=[0,\,2L], \quad \textit{and} \quad supp(\Omega_r)=[-2L,\,0]. $$
For the approximation of $\Omega_\ell$ we collocate \eqref{Mar_K} in the nodal points
$$\{(x_{n-2i}, x_n), \quad x_{n-2i}=(n-2i)h, \quad i=0,1,...,n \},$$
by obtaining
\begin{equation}
\Omega_\ell(x_{2(n-i)})+\int_{x_{n-2i}}^{x_{n}} K^{\dn}(x_{n-2i},z) \Omega_\ell(z+x_n) dz=-\bar K^{\dn}(x_{(n-2i)},x_n).
\end{equation}
Now, to compute the above integral  we use different quadrature formula by adopting a steplenght $\delta=2h$ that is twice the one considered in the numerical solution of system \eqref{nucleiK} to avoid the interpolation among the values of the auxiliary functions computed before. 
More precisely, for
\begin{itemize}
\item[$\bullet$] $i=0$ we immediately obtain that
$$\Omega_{\ell,2n}=-\bar{K}^{\dn}_{n,n}=-\frac{1}{2} v_n $$
in virtue of \eqref{propr1};
\item[$\bullet$] $i=1$ we use the trapezoidal rule by getting
$$\left(1+\frac{\delta}{2} K^{\dn}_{n-2,n-2}\right)\Omega_{\ell,2(n-1)}=-\bar{K}^{\dn}_{n-2,n}-K^{\dn}_{n-2,n} \Omega_{\ell,2n};$$
\item[$\bullet$] $i=2,4,6,...$ we use the Simpson quadrature formula
\begin{align*}
\left(1+\frac{\delta}{3} K^{\dn}_{n-2i,n-2i}\right)\Omega_{\ell,2(n-i)}=-\bar{K}^{\dn}_{n-2i,n}&-\frac{\delta}{3} \left(4 \sum_{j=1}^{\frac{i+1}{2}} K^{\dn}_{n-2i,n-2(i-j)} \Omega_{\ell,2(n-2(i-j))} \right.\\ & \hspace{-4cm} \left.+2\sum_{j=1}^{\frac{i+1}{2}-1} K^{\dn}_{n-2i,n-2(i-j-1)} \Omega_{\ell,2(n-(i-j-1))} + K^{\dn}_{n-2i,n} \Omega_{\ell,2n} \right); 
\end{align*}

\item[$\bullet$] $i=3,5,7,...$  as we can write
\begin{align*}
&\int_{x_{n-2i}}^{x_{n}} K^{\dn}(x_{n-2i},z) \Omega_\ell(z+x_n) dz \\ &= \left\{ \int_{x_{n-2i}}^{x_{n-2(i-3)}}+ \int_{x_{n-2(i-3)}}^{x_n} \right\} K^{\dn}(x_{n-2i},z)  \Omega_\ell(z+x_n) dz
\end{align*}
we approximate the first integral by using the $3/8$ Simpson rule and the last integral by adopting the composite Simpson quadrature formula. Hence we get
\begin{align*}
 \left(1+\frac{3 \delta}{8} K^{\dn}_{n-2i,n-2i}\right) \Omega_{\ell,2(n-i)}  &=-\bar{K}^{\dn}_{n-2i,n} - \frac{3 \delta}{8}  \left( 3 K^{\dn}_{n-2i,n-2(i-1)} \Omega_{\ell,2(n-(i-1))} \right. \\ & \hspace{-3cm}\left.+ 3 K^{\dn}_{n-2i,n-2(i-2)} \Omega_{\ell,2(n-(i-2))}+ K^{\dn}_{n-2i,n-2(i-3)} \Omega_{\ell,2(n-(i-3))}  \right) \\ & \hspace{-3cm} - \frac{\delta}{3} \left( K^{\dn}_{n-2i,n-2(i-3)} \Omega_{\ell,2(n-(i-3))}+ 4 \sum_{j=1}^{\frac{i+1}{2}} K^{\dn}_{n-2i,n-2(2i-3-j)} \Omega_{\ell,2(n-(2i-3-j))} \right. \\ & \hspace{-3cm} \left. +  2 \sum_{j=1}^{\frac{i+1}{2}-1} K^{\dn}_{n-2i,n-2(2i-4-j)} \Omega_{\ell,2(n-(2i-4-j))} +K^{\dn}_{n-2i,n} \Omega_{\ell,2n} \right).
\end{align*}
\end{itemize}  
An analogous procedure can be applied to  approximate $\Omega_r$ in $[-2L,\,0]$. More precisely, 
we collocate \eqref{Mar_M} in the nodal points
$$\{(x_{2i-n}, x_{-n}), \quad x_{2i-n}=(2i-n)h, \quad i=0,1,...,n \},$$
by obtaining
\begin{equation}
\Omega_r(x_{2(i-n)})+\int_{x_{-n}}^{x_{2i-n}} M^{\up}(x_{2i-n},z) \Omega_\ell(z+x_{-n}) dz=-\bar M^{\up}(x_{(2i-n)},x_{-n}).
\end{equation}
Hence, by adopting the technique illustrated above, 
\begin{itemize}
\item[$\bullet$] for $i=0$ we immediately obtain
$$\Omega_{r,-2n}=-\bar{M}^{\up}_{-n,-n}=-\frac{1}{2} u_{-n} $$
in virtue of \eqref{propr1};
\item[$\bullet$] for $i=1$ we obtain
$$\left(1+\frac{\delta}{2} M^{\up}_{2-n,2-n}\right)\Omega_{r,2(1-n)}=-\bar{M}^{\up}_{2-n,-n}-M^{\up}_{2-n,-n} \Omega_{r,-2n};$$
\item[$\bullet$] for $i=2,4,6,...$ we obtain 
\begin{align*}
\left(1+\frac{\delta}{3} M^{\up}_{2i-n,2i-n}\right)\Omega_{r,2(i-n)}=-\bar{M}^{\up}_{2i-n,-n}&-\frac{\delta}{3} \left(4 \sum_{j=1}^{\frac{i+1}{2}} M^{\up}_{2i-n,2(i-j)-n} \Omega_{r,2((i-j)-n)} \right.\\ & \hspace{-4cm} \left.+2\sum_{j=1}^{\frac{i+1}{2}-1} M^{\up}_{2i-n,2(i-j-1)-n} \Omega_{r,2((i-j-1)-n)} + M^{\up}_{2i-n,-n} \Omega_{r,-2n} \right);
\end{align*}
\item[$\bullet$] for $i=3,5,7,...$  as we can write
\begin{align*}
&\int_{x_{-n}}^{x_{2i-n}} M^{\up}(x_{2i-n},z) \Omega_r(z+x_{-n}) dz \\ &= \left\{ \int_{x_{-n}}^{x_{2(i-3)-n}}+ \int_{x_{2(i-3)-n}}^{x_{2i-n}} \right\} M^{\up}(x_{2i-n},z)  \Omega_r(z+x_{-n}) dz,
\end{align*}
we approximate the first integral by using the composite  Simpson rule and the second one by adopting the $3/8$ Simpson's quadrature formula. Hence we get
\begin{align*}
 \left(1+\frac{3 \delta}{8} M^{\up}_{2i-n,2i-n}\right) \Omega_{r,2(i-n)}  &=-\bar{M}^{\up}_{2i-n,-n} - \frac{3 \delta}{8}  \left( 3 M^{\up}_{2i-n,2(i-1)-n} \Omega_{r,2((i-1)-n)} \right. \\ & \hspace{-3cm}\left.+ 3 M^{\up}_{2i-n,2(i-2)-n} \Omega_{r,2((i-2)-n)}+ M^{\up}_{2i-n,2(i-3)-n} \Omega_{r,2((i-3)-n)}  \right) \\ & \hspace{-3cm} - \frac{\delta}{3} \left( M^{\up}_{2i-n,2(i-3)-n} \Omega_{r,2((i-3)-n)}+ 4 \sum_{j=1}^{\frac{i+1}{2}} M^{\up}_{2i-3,2(2i-3-j)-n} \Omega_{r,2((2i-3-j)-n)} \right. \\ & \hspace{-3cm} \left. +  2 \sum_{j=1}^{\frac{i+1}{2}-1} M^{\up}_{n-2i,2(2i-4-j)-n} \Omega_{r,2((2i-4-j)-n)} +M^{\up}_{2i-n,-n} \Omega_{r,-2n} \right).
\end{align*}
\end{itemize}

\subsection{Computation of the scattering matrix and inverse Fourier transforms of reflection coefficients} 
In this section we illustrate our method to approximate the scattering matrix and then to compute the transmission coefficients $T$ defined in \eqref{T}, the reflection coefficients $R$ and $L$ introduced in \eqref{R}-\eqref{L} and their Fourier transforms $\rho$ and $\ell$ given in \eqref{rho}-\eqref{elle}, under the assumption that $u_0 \in C(\mathbb{R})$. 
\subsection*{Approximation of the transmission coefficient T}
It is based on the two equivalent definitions of the transmission coefficient 
\begin{equation}\label{T1}
T(\lambda)= \dfrac{1}{a_{\ell 4}(\lambda)}, \qquad T(\lambda)= \dfrac{1}{a_{r 1}(\lambda)}
\end{equation}
that is on the computation of the coefficients of the transition matrices
\begin{align}
a_{\ell4}(\lambda)&= 1+\displaystyle\int_{\R^+} e^{i \lambda z} {\Phi}^{\up}(z) dz=1+2 \pi \iftran{\Phi^{\up}(\lambda) H(\lambda)}, \label{al4}\\
a_{r1}(\lambda)&= 1+\displaystyle\int_{\R^+} e^{i \lambda z} \Psi^{\dn}(z) dz=1+2 \pi \iftran{\Psi^{\dn}(\lambda) H(\lambda)} \label{ar1},
\end{align}
where $H$ denotes the Heaviside function and $\iftran{g}$ stands for the inverse Fourier transform of $g$.

Let us only illustrated the algorithm for the computation of the coefficient $a_{\ell 4}$ as the computation of $a_{r1}$ is analogous.

At first we note that,  taking into account \eqref{assu0} and the support of $K^{\up}$, the kernel $\Phi^{\up}$ of \eqref{eq_rho1} can be written as follows:
\begin{equation}
\Phi^{\up}(z)=
\begin{cases}
\displaystyle \int_{-L}^{L-\frac{z}{2}} v_0(y) K^{\up}(y,y+z) dz, & for \quad  0 \leq z \leq 4L\\
0, & for \quad z>4L.
\end{cases}
\end{equation}

Then, writing, for simplicity,
$\Phi^{\up}_j=\Phi^{\up}(z_j)=\Phi^{\up}(2 h j), \, j=0,1,2,...,2n-1$
we have successively to compute $\Phi^{\up}_0, \, \Phi^{\up}_1,\, ...,\, \Phi^{\up}_{2n-1}$ by obtaining
$$\Phi^{\up}_j=\int_{-n h}^{(n-j)h} v_0(y) K^{\up}(y, y+2hj) dy, \quad j=0,1,...,2n-1.$$
We remark that its computation requires only the values of $K^{\up}(y,y+2hj)$ which we have already computed since they are the values of $K^{up}$ on the $j$th parallel to the bisector $y=x$.  For this reason $\Phi^{\up}_j$ can be computed by simply adopting the computational strategy that we developed for computing $K^{\up}$. At this point the approximation of $T(\lambda)$, easily follows by using \eqref{T1}.  
\subsection*{Approximation of the reflection coefficients  $R$ and $L$}
In the matter of the computation of the reflection coefficients, taking into account \eqref{R} and \eqref{L}, we can write
\begin{equation}\label{R1}
R(\lambda)= - T(\lambda) \, a_{\ell 3}(\lambda), \qquad L(\lambda)=  T(\lambda) \, a_{\ell 2}(\lambda) 
\end{equation}
where $T(\lambda)= \dfrac{1}{a_{\ell 4}(\lambda)}$,
\begin{align*}
a_{\ell3}(\lambda)&=\displaystyle \int_{\R} e^{-i \lambda y}  \left[ \frac{1}{2} v_0\left(\frac{y}{2}\right)+\bar{\Phi}^{\up}(y) \right] \, dy = \ftran{\frac{1}{2} v_0\left(\frac{y}{2}\right)+\bar{\Phi}^{\up}(y)},
\end{align*}
and
\begin{align*}
a_{\ell2}(\lambda)&=-\displaystyle\int_{\R} e^{i \lambda y} \left( \frac{1}{2} u_0 \left(\frac{y}{2}\right)+\Phi^{\dn}(y) \right)  dy = - 2 \pi \iftran{\frac{1}{2} u_0\left(\frac{y}{2}\right)+\Phi^{\dn}(y)}.
\end{align*}
Other equivalent expressions can be deducted by using the definitions of $R$, $L$ and $T$ in terms of the coefficients of the transmission matrix from the right.

To approximate $a_{\ell 3}$, taking into account \eqref{assu0} and the support of $\bar{K}^{\up}$, first we note that
\begin{equation}
\bar{\Phi}^{\up}(z)=
\begin{cases}
\displaystyle \int_{-L}^{\frac{z}{2}} v_0(y) \bar{K}^{\up}(y,z-y) dy, & for \quad  |z| \leq 2L \\
0, & for \quad |z|>2L.
\end{cases}
\end{equation}

Moreover, adopting the notation used before and noting that $\bar{\Phi}^{\up}_{-n}=\bar{\Phi}^{\up}(-n)=\bar{\Phi}^{\up}(-2 n h)=0,$
we can write $$\bar{\Phi}^{\up}_i=
\int_{-n h}^{i h} v_0(y) \bar{K}^{\up}(y, 2 h i-y) dy, \quad i=-n+1,...,\,0,\,...,\,n. $$ 
Hence $\bar{\Phi}^{\up}_i$, as well as $\Phi^{\up}_j$, can be computed by simply adapting the computational strategy developed for $K^{\up}$. The approximation of $R$ and $L$ immediately follow by using \eqref{R1}.

\subsection{Computation of the bound states and the norming constants}\label{sect:Prony}
For the sake of completeness, we now give a brief description of the matrix-pencil method that we have recently developed for the identification of the bound states and the norming constants 
\cite{Fermo2014-AMC,Fermo2014-ETNA}.
Setting $z_j=e^{i \lambda_j}$, the spectral function sum $S_\ell(\alpha)$ introduced in \eqref{S_ell} can be represented as the monomial-power sum
\begin{equation*}
S_\ell(\alpha)=\sum_{j=1}^{n} \sum_{s=0}^{m_j-1} c_{js} \alpha^s z_j^\alpha, \quad 0^0 \equiv 1.
\end{equation*}

Letting $M=m_1+...+m_n$, the method allows one to compute the parameters $\{n,m_j,z_j\}$ and the coefficients $\{c_{js}\}$,  given $S_\ell(\alpha)$ in $2N$ integer values ($N>M$) 
$$\alpha=\alpha_0,\alpha_0+1,\dots, \alpha_0+2N-1, \quad \text{with} \quad \alpha_0 \in \mathbb{N}^+=\{0,1,2,... \},$$ 
under the assumption that a reasonable overestimate of $M$ is known.

The basic idea of the method is the interpretation of $S_\ell(\alpha)$  as the general solution of a homogeneous linear difference equation of order $M$
$$ \sum_{k=0}^M p_k S_{k+\alpha_0}=0$$
whose characteristic polynomial (Prony's polynomial)
$$P(z)=\prod_{j=1}^{n} (z-z_j)^{m_j} =\sum_{k=0}^M p_k z^k, \quad p_M \equiv 1$$
is uniquely characterized by the $z_j$ values we are looking for. The identification of the zeros $\{ z_j\}$ allows one to compute the  coefficients $c_{js}$ by solving in the least squares sense a linear system.

For the computation of $\{z_j\}$ and then of the bound states $\lambda_j$, the given data are arranged in the two Hankel matrices of order $N$
$$({\bf{S}}_{\ell}^{0})_{ij}= S_\ell(i+j-2), \quad ({\bf{S}}_{\ell}^{1})_{ij}=S_\ell(i+j-1), \quad \quad i,j=1,2,\dots,N. $$

To these matrices we then  associate the $M \times M$ matrix-pencil
$$ {\bf{S}}_{MM}(z)=({{\bf{S}}_{NM}^{0}})^* ({\bf{S}}_{NM}^{1} - z {\bf{S}}_{NM}^{0})$$
where the asterisk denotes the conjugate transpose.
As proved in \cite{Fermo2014-ETNA}, the zeros $z_j$ of the Prony polynomial, with their multiplicities, are exactly the generalized eigenvalues of the matrix-pencil ${\bf{S}}_{MM}(z)$. The simultaneous factorization of the matrices  ${{\bf{S}}_{NM}^{0}}$ and ${{\bf{S}}_{NM}^{1}}$ by the Generalized Singular Value Decomposition allows us to compute the zeros $z_j$ and then the bound states $\lambda_j$, as $\lambda_j=-i \log{z_j}.$ 

Analogous results can be obtained by a proper factorization of the augmented Hankel matrix ${{\bf{S}}_{\ell}}=[{{\bf{S}}_{\ell, 1}^{0}}, { {\bf{S}}_{\ell}^{1}}]$, where ${{\bf{S}}_{\ell, 1}^{0}}$ is the first column of ${{\bf{S}}_{NM}^{0}}$ and ${{\bf{S}}_{\ell}^{0}}$ is obtained by ${{\bf{S}}_{\ell}^{1}}$ by simply deleting its last column. As shown in \cite{Fermo2014-AMC}, the QR factorization of ${{\bf{S}}_{\ell}}$ is as effective as its SVD factorization considered in \cite{Fermo2014-ETNA}, though its computational complexity is generally smaller.

The vector of coefficients  $${\bf{c}}=[c_{1\,0},...,c_{1\,n_1-1},...,c_{L\,0},...,c_{L\,n_{M}-1}]^T$$ is then  computed by solving (in the least square sense) the overdetermined linear system 
\begin{equation*}
\mathbf{K}_{NM}^{0} \mathbf{c}=\mathbf{S}_\ell^{0}
\end{equation*}
where ${\bf{S}_\ell}^{0}=[S_\ell(0),\, S_\ell(1),\, \dots,\, S_\ell(N-1)]^T$
and  $\mathbf{K}_{NM}^{0}$ is the Casorati matrix associated to the monomial powers $\{k^s z_j^k\}$ for $k=1,...,N-1$.

If $m_j \equiv 1$, the Casorati matrix ${\bf K}^{0}_{Nn}$ reduces to the Vandermonde matrix $(V)_{ij}=z_j^{i+1}$ of order $N \times n$ associated to the  zeros $ z_1,\, \dots,\, z_n$.
The solution of the Casorati system allows us to immediately compute the norming constants as $(\Gamma_\ell)_{js}=s! c_{js}.$

The coefficients $\{ (\Gamma_r)_{js} \}$ are then obtained by solving, in the least square sense, a linear system whose vector of known data is given by $\Omega_r(\alpha)$ evaluated in a set of $N$ points, with a sufficiently large $N>M$.
\section{Examples}
\label{sect:examples}
Let us now present two examples. The first one is a reflectionless case while the second one has  reflection coefficients different from zero. Each of them will be used in the next section to give a numerical evidence of the effectiveness of our method.
\section*{Example 1 (One soliton potential)} 
Considering the initial potential for the NLS in the focusing case we take  
\begin{equation}\label{onesoliton}
u_0(x)=2 \im \eta e^{\im (2\xi x+\phi)} 
\sech{(x_0-2\eta x)}
\end{equation}
where $\xi,\phi, x_0 \in \mathbb{R}$ and $0 \neq \eta \in \mathbb{R}$. As proved in \cite{HasTap1}, the corresponding initial value problem \eqref{NLS} can be solved exactly, as already considered in several papers and in particular in \cite{ArVanSe2009} and  \cite{ArRoSe2011}.
Let us note that $2 \eta>0$ represents the amplitude of the initial potential and $\mu_0=x_0/2\eta$ is the initial peak position. 

In this example the norming constants from the left and from the right are \cite{ArRoSe2011}:
\begin{equation}
\Gamma_\ell=2 \im \eta e^{x_0-\im \phi} \quad\text{and} \quad  \Gamma_r=-2 i \eta e^{-x_0+\im \phi}.
\end{equation}

Moreover, setting   $a=\eta+\im \xi$, as it is immediate to check, the exact solution of the Volterra system \eqref{nucleiK} for $y \geq x$ is 
\begin{equation*}
\left( \begin{matrix}
K^{\up}(x,y) \\ \\  K^{\dn}(x,y)
\end{matrix} \right)= - \dfrac{1}{1+e^{2(x_0-2 \eta x)}} \left( \begin{matrix}
-\Gamma_\ell^* \, e^{-a^*(x+y)} \\ \dfrac{|\Gamma_\ell|^2}{2 \eta} e^{-a^*(x+y)-2ax}
\end{matrix} \right),
\end{equation*}
while the exact solution of system \eqref{nucleiKbarrati} can be obtained by resorting to relation \eqref{symmpropr1}.
Furthermore, the closed form solution of the Volterra system \eqref{nucleiM} is
\begin{equation*}
\left( \begin{matrix}
M^{\up}(x,y) \\ \\  M^{\dn}(x,y)
\end{matrix} \right)=  - \dfrac{1}{1+e^{-2(x_0-2 \eta x)}} \left( \begin{matrix}
\dfrac{|\Gamma_r|^2}{2 \eta} e^{a^*(x+y)+2ax}  \\  -\Gamma_r^* \, e^{a^*(x+y)}
\end{matrix} \right),
\end{equation*}
while the solution of system \eqref{nucleiMbarrati} can be deducted by using relation  \eqref{symmpropr1}. 
As it represents a reflectionless case, 
$$\rho(\alpha)=\ell(\alpha)=0, \quad \alpha \in \mathbb{R},$$
so that the exact initial Marchenko kernels are \cite{ArRoSe2011}
\begin{align*}
\Omega_\ell(x)&=\Gamma_\ell e^{-ax}, \\
\Omega_r(x)&=\Gamma_r e^{ax}.
\end{align*}
Finally, the scattering matrix is
\begin{equation*}\label{S}
{\bf S}(\lambda)=\begin{pmatrix}T(\lambda)&L(\lambda)\\ R(\lambda)& T(\lambda)
\end{pmatrix}=\begin{pmatrix}\dfrac{\lambda+ia^*}{\lambda-ia^*}& 0\\ 0 & \dfrac{\lambda+ia^*}{\lambda-ia^*}
\end{pmatrix}, \qquad  \lambda \in \mathbb{C}^+ .
\end{equation*}

\section*{Example 2 (Gaussian potential)}
As a second example of the initial potential for the NLS, we take
\begin{equation}\label{gaussiana}
u_0(x)=q_0 e^{\im \mu x} e^{-\frac{x^2}{\sigma}},
\end{equation}
where $q_0>0, \sigma>0$ and $\mu \in \mathbb{R}$. 

As in \cite{Osborne1993,Wahls2013} we investigate the defocusing case in which the scattering coefficients $T(\lambda)$, $R(\lambda)$ and $L(\lambda)$ are all continuous functions and there are no bound states. Hence, in this case the following relations hold true 
\begin{equation}\label{nobound}
\Omega_\ell(\alpha) \equiv \rho(\alpha) \quad \Omega_r(\alpha) \equiv \ell(\alpha).
\end{equation}
Moreover, we also consider the focusing case. In such a case, whenever $$q_0 \sqrt{\pi \sigma} < \frac{\pi}{2},$$ there are no discrete eigenvalues. On the contrary we have $n$  discrete eigenvalues, all of them simple and having real part $-\frac{\mu}{2}$, if \cite{KS} 
\begin{equation}\label{bound_gaussian}
\left(n-\frac{1}{2}\right) \pi<q_0 \sqrt{\pi \sigma}<\left(n+\frac{1}{2}\right)\pi.
\end{equation}
As a result the spectral sums from the left and from the right \eqref{S_ell} and \eqref{S_r} reduce to 
\begin{align}
S_\ell(\alpha)&=\sum_{j=1}^n (\Gamma_\ell)_j e^{i \lambda_j \alpha} \quad \alpha>0\\
S_r(\alpha)&=\sum_{j=1}^n (\Gamma_r)_j e^{i \lambda^*_j \alpha}, \quad \alpha<0.
\end{align}
We also remark that the reflection coefficients $R(\lambda)$ and $L(\lambda)$ and the transmission coefficient $T(\lambda)$ are discontinuous at $\lambda=-\frac{\mu}{2}$ if \cite{KS} $$q_0 \sqrt{\pi \sigma}=\left(n-\frac{1}{2}\right) \pi$$ for some positive integer $n$.


\section{Numerical results and conclusions}
\label{sect:tests}

\section*{Test 1 (One soliton potential)}
Let us consider as in \cite{ArRoSe2011} the initial potential \eqref{onesoliton} with $\xi=1/10$, $x_0=\phi=0$ and $\eta=2$. In order to compute the non-zero scattering parameters that in this case are the norming constants, the bound states and the transmission coefficient, at first we solve the Volterra's system \eqref{nucleiK} and \eqref{nucleiM} with
$L=8$ and $n=3000$ by obtaining the following relative errors
\begin{align*}
\frac{\|K^{\up}-\tilde{K}^{\up}\|}{\|K^{\up}\|} = 1.80e-06, \quad \frac{\|K^{\dn}-\tilde{K}^{\dn}\|}{\|K^{\dn}\|} = 1.04e-07, \\
\frac{\|M^{\up}-\tilde{M}^{\up}\|}{\|M^{\up}\|} =  1.07e-07, \quad 
\frac{\|M^{\dn}-\tilde{M}^{\dn}\|}{\|M^{\dn}\|} =  1.80e-06,
\end{align*}
where here and in the sequel the $\sim$ sign denotes the approximation of the exact function previously given and $\| \cdot\|$ denotes the maximum norm of the involved function in their computational areas.
Identical relative errors are of course obtained for the remainding auxiliary functions, as a result of the symmetry properties  \eqref{symmpropr1} and \eqref{symmpropr2}.

Once these auxiliary functions are computed we numerically solve equations getting for the Marchenko kernels from the right and from the left with the following relative errors:
\begin{align*}
\dfrac{\displaystyle \max_{x \in [0,2L]}{|\tilde{\Omega}_\ell(x)-\Omega_\ell}(x)|}{\displaystyle \max_{x \in [0, 2L]}{|\Omega_\ell(x)|}}& \simeq \dfrac{\displaystyle \max_{x \in [-2L,0]}{|\tilde{\Omega}_r(x)-\Omega_r}(x)|}{\displaystyle \max_{x \in [-2L, 0]}{|\Omega_r(x)|}} \simeq  3.24e-07,
\end{align*}
where the symbol $\simeq$ means that the left term coincide with the right term up to the third decimal digit.

At this point, by using such kernels, we apply our matrix pencil method \cite{Fermo2014-AMC} by finding a single bound state term, a norming constant from the left and a norming constant from the right with the following relative errors: 
\begin{align*}
\frac{|\tilde{\lambda}-\lambda|}{|\lambda|} =4.11e-09, \quad \frac{|\tilde{\Gamma_\ell}-\Gamma_\ell|}{|\Gamma_\ell|} \simeq \frac{|\tilde{\Gamma_r}-\Gamma_r|}{|\Gamma_r|} \simeq 3.24e-07.
 \quad 
\end{align*}

In the matter of the relative errors of the scattering matrix, we obtain 

\begin{align*}
\dfrac{\displaystyle \max_{\lambda \in [-2L,2L]}{\|\tilde{\mathbf{S}}(\lambda)-\mathbf{S}(\lambda)\|}}{\displaystyle \max_{\lambda \in [-2L,2L]}{\|\mathbf{S}(\lambda)\|}}=  4.60e-07.
\end{align*}
Moreover to ascertain the effectiveness of our  numerical method we checked the numerical validity of the algebraic property \eqref{proprS2}. The results are at all satisfactory as  Figure \ref{graficonorma1} shows where the behavior of the error function
$$E_s(\lambda)=\left \| \dfrac{1}{2}(\mathbf{S}^\dagger(\lambda) \mathbf{J} \mathbf{S} (\lambda)+\mathbf{S}(\lambda) \mathbf{J} \mathbf{S}^\dagger(\lambda))-\mathbf{J} \right \| $$
 is reported for $\lambda \in [-2L,2L]$  in semilog scale.

\begin{figure}[h]
\includegraphics[scale=0.4]{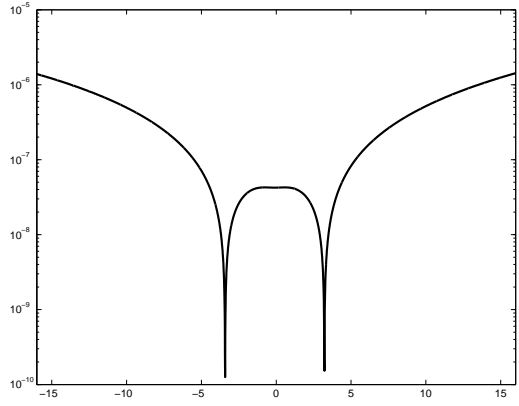}
\caption{$E_s(\lambda)$ in semi logarithmic scale \label{graficonorma1}}
\end{figure}

Concerning the trasmission coefficient, we can compute it by approximating at first the integral $\Phi^{\up}$ defined in \eqref{Phi2}, and then using \eqref{T1}. In Table \ref{table1} we give the following relative errors we obtain for such a coefficient over  segment of width $4L$ of three different lines 
$$E_r(T)=\frac{\displaystyle \max_{\lambda \in [a,b]}{|\tilde{T}(\lambda)-T(\lambda)|}}{\displaystyle \max_{\lambda \in [a, b]}{|T(\lambda)|}}.$$

\begin{table}[h]
\caption{$E_r(T)$ in the one soliton case \label{table1}}
\begin{tabular}{ll}
\hline\noalign{\smallskip}
 $[a,b]$ & $E_r(T)$  \\
 \noalign{\smallskip}\hline\noalign{\smallskip}
 $[-2L,2L]$ & $3.13e-07$ \\
 $[-2L+\mathbf{i},2L+\mathbf{i}]$ & $2.21e-07$ \\
 $[-2L+5\mathbf{i},2L+5\mathbf{i}]$ & $ 4.47e-07$ \\
\noalign{\smallskip}\hline
\end{tabular}
\end{table}

\section*{Test 2 (Gaussian potential)}
Let us consider first the initial potential \eqref{gaussiana} in the defocusing case with $q_0=1.9,\,\mu=1,\,\sigma=2$ as in \cite{Osborne1993,Wahls2013}. To this end, we compute the solution of  systems \eqref{nucleiKbarrati}-\eqref{nucleiM} considering as in the soliton case $L=8$ and $n=3000$, then we solve equations \eqref{Mar_K}-\eqref{Mar_M}, compute the scattering matrix and thus the Fourier transforms of the reflection coefficients. Our numerical method recognizes that, as 
theoretically expected, there are no bound states and relations \eqref{nobound} are numerically satisfied since we have the following errors:
\begin{equation*}
\max_{x \in [0,2L]}|\Omega_\ell(x)-\rho(x)| = 1.08e-10, \quad \max_{x \in [-2L,0]}|\Omega_r(x)-\ell(x)| =  1.44e-09.
\end{equation*}
As in the one soliton case, we checked if our numerical results satisfy the algebraic property \eqref{proprS2} for the scattering matrix, by considering in semi logarithmic scale the error function 
$$E_{GD}(\lambda)=\left \| \dfrac{1}{2}(\mathbf{S}^\dagger(\lambda) \mathbf{S} (\lambda)+\mathbf{S}(\lambda) \mathbf{S}^\dagger(\lambda))-\mathbf{I} \right \| $$
for $\lambda \in [-2L,2L]$. As shown in Figure \ref{graficonorma} its numerical validity is satisfactory as in the soliton case. 
 
\begin{figure}[h]
\includegraphics[scale=0.33]{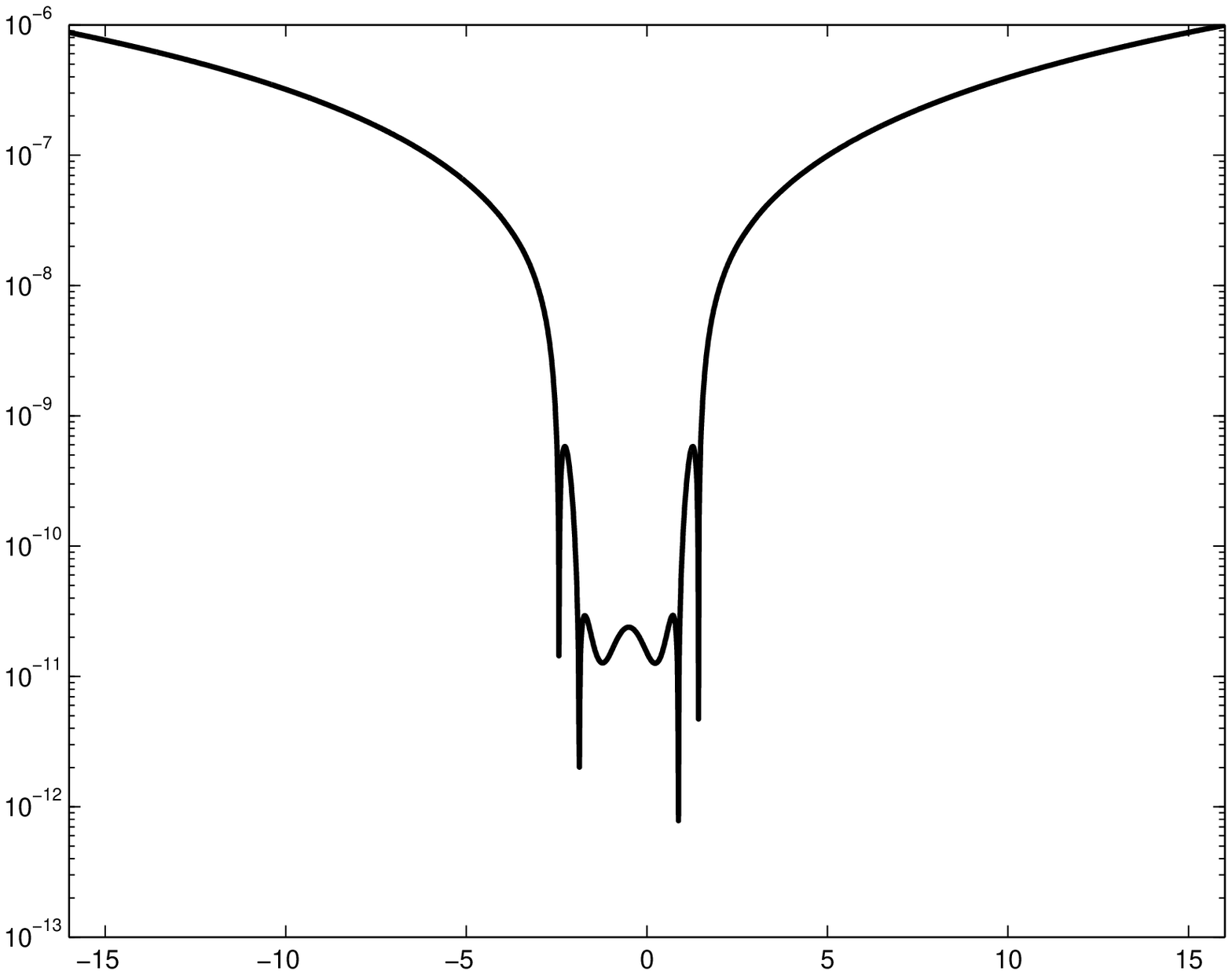}
\includegraphics[scale=0.33]{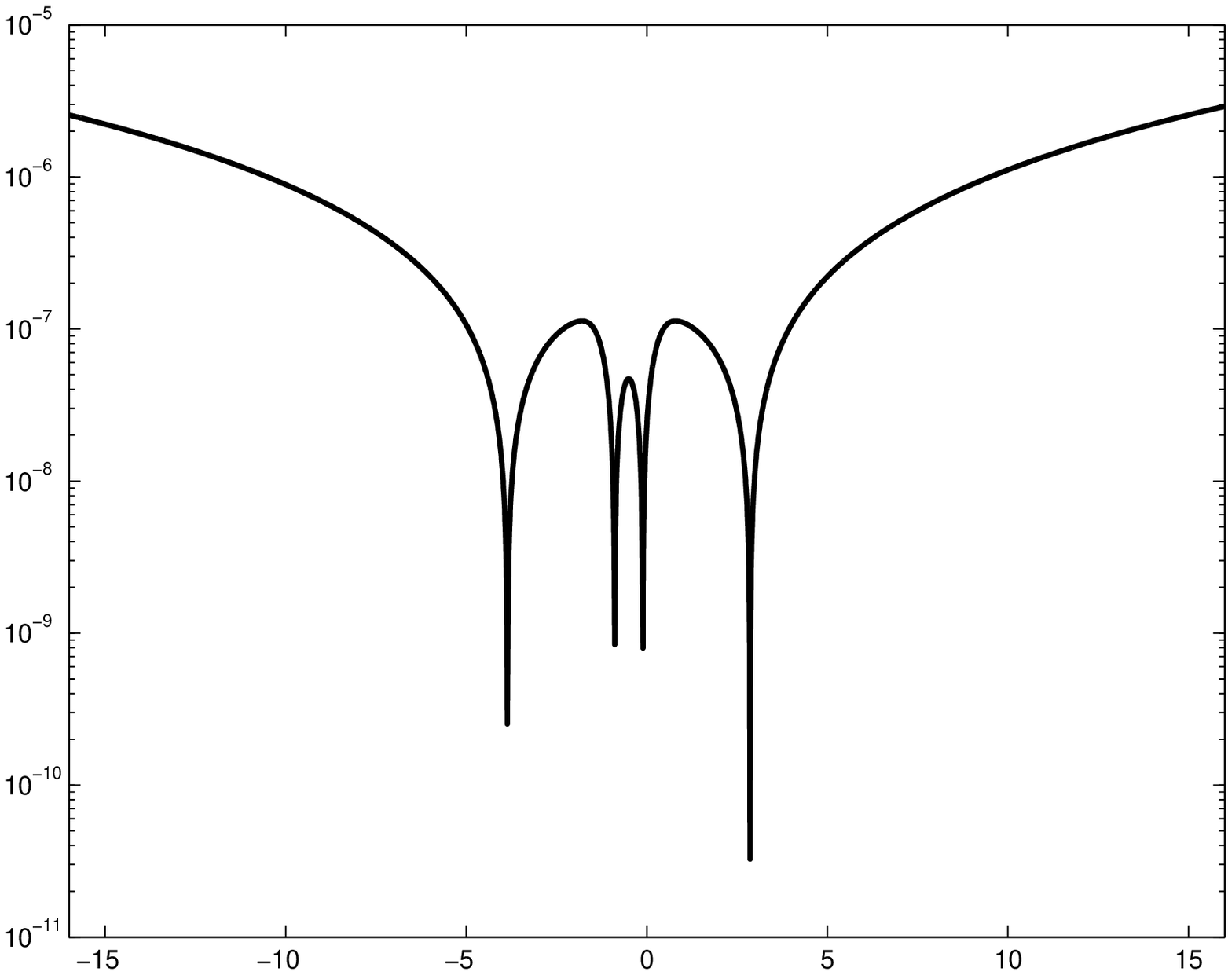}
\caption{$E_{GD}$ (to the left) and $E_{GF}$ (to the right) in semi logarithmic scale  \label{graficonorma}}
\end{figure}

Now let us investigate on the focusing case  considering the initial potential \eqref{gaussiana} with $q_0=2.5,\,\mu=1,\,\sigma=2$. As a result, inequality \eqref{bound_gaussian} implies that we have two simple bound states $\{\lambda_1,\lambda_2\}$ whose real part is $-1/2$. At first we compute the auxiliary functions by solving systems \eqref{nucleiKbarrati}-\eqref{nucleiM} with $L=8$ and $n=3000$, then we solve equations \eqref{Mar_K}-\eqref{Mar_M}, compute the scattering matrix and  the Fourier transforms of the reflection coefficients. 
At this point, we apply the matrix pencil method described in Section \ref{sect:Prony} assuming that we have not more than five bound states. Our method recognize that, as theoretically expected, we have two simple bound states having  real part equal to $-\mu/2$. In fact we get
\begin{align*}
\lambda_1&=-0.50+1.97 \im  \quad \lambda_2=-0.50+0.79\im
\end{align*}
with the corresponding norming constants 
\begin{align*}
\Gamma_{\ell,1}&= 9.28-1.50 \, 10^{-8}\im \quad 
\Gamma_{\ell,2}= 3.74-1.76 \,  10^{-11} \im \\
\Gamma_{r,1}&= 9.28+1.50 \, 10^{-8}\im \quad 
\Gamma_{r,2}= 3.74+1.76 \,  10^{-11} \im.
\end{align*}

Finally, in Figure \ref{graficonorma} we represent in semi logarithmic scale the error function
$$E_{GF}(\lambda)=\left \| \dfrac{1}{2}(\mathbf{S}^\dagger(\lambda) \mathbf{J} \mathbf{S} (\lambda)+\mathbf{S}(\lambda) \mathbf{J} \mathbf{S}^\dagger(\lambda))-\mathbf{J} \right \| $$
for $\lambda \in [-2L,2L]$ that we have computed to check the validity of the algebraic property \eqref{proprS1}. 

\section*{Conclusions}
The numerical results show that our numerical method is effective in both the focusing and defocusing cases, provided the initial potential decays to zero at infinity and is at least continuous. This positive result is due to the possibility to know each pair of functions on the whole plane, by solving the relative Volterra system on a bounded computational triangle. The accuracy of the identification of the spectral parameters strongly depends on this result, since all the subsequent computations require the knowledge of the auxiliary functions on their computational triangles. 

We believe that the method can be extended, with the same accuracy of the results, in the presence of jump discontinuities of the initial potential. To this end, a numerically stable method for the solution of Fredholm integral equations \eqref{eq_rho1}-\eqref{eq_rho2} and \eqref{eq_elle1}-\eqref{eq_elle2} should be developed. The development of such a method should also be accompanied by an extensive numerical experimentation which requires the exact knowledge of scattering data in at least one case in which the initial potential has jump discontinuities. Considering that such research takes a rather long time, the development of such a method is postponed to a next paper.     

\section{Appendix} \label{sect:appendix}
\section*{Supports of the auxiliary functions}

In this section we determine the supports of the auxiliary functions $K(x,y)$
and $M(x,y)$ if the potentials $u_0(x)$ and $v_0(x)$ have their supports in
$[-L,L]$. It suffices to prove parts (2) of Lemmas 5.1 and 5.2 in \cite{Fermo2014-PUB}, because the
proofs of the other three parts of these two lemmas are immediate and proceed
as in the discrete case.

Put
$$\nu(\bar{K}^{\up};x)
=\int_x^\infty|\bar{K}^{\up}(x,y)|\,dy,\qquad
\nu(\bar{K}^{\dn};x)
=\int_x^\infty|\bar{K}^{\dn}(x,y)|\,dy;$$
$$Q(x)=\max(|u_0(x)|,|v_0(x)|),\qquad P(x)=\nu(\bar{K}^{\up};x)
+\nu(\bar{K}^{\dn};x),$$
where $Q$ and $P$ are bounded \cite{VanDerMee2013}.
Then for $x\le L$ and $x+y\ge2L$ the integral equations (3.1) have zero
right-hand sides, because $v_0(\tfrac{1}{2}(x+y))=0$ for $x+y>2L$. Integrating
the absolute values of $\bar{K}^{\up}(x,y)$ and
$\bar{K}^{\text{\tiny \dn}}(x,y)$ with respect to $y\in(x,+\infty)$, we
obtain
\begin{align*}
\nu(\bar{K}^{\up};x)&\le\int_x^L|u_0(z)|
\nu(\bar{K}^{\dn};z)\,dz,\\
\nu(\bar{K}^{\dn};x)&\le\int_x^L|v_0(z)|
\nu(\bar{K}^{\up};z)\,dz,
\end{align*}
so that
$$0\le P(x)\le\int_x^L Q(z)P(z)\,dz.$$
Hence iterating two times  the last inequality we have
\begin{align*}
P(x)\le\int_x^L Q(z)P(z)\,dz & \leq \int_x^L Q(z)  \int_z^L Q(t) \int_t^L Q(w) P(w) \, dw \, dt \, dz \\ & \leq \left( \int_x^L Q(w) P(w) \, dw \right) \left( \int_x^L Q(z) \int_z^L Q(t) \, dt \, dz \right) \\ & = \left( \int_x^L Q(w) P(w) \, dw \right) \left( \int_x^L  -\frac{1}{2} \frac{d}{dz} \left( \int_z^L Q(t) \, dt \right)^2 \, dz \right) \\ &= \left( \int_x^L Q(w) P(w) \, dw \right) \left[ -\frac{1}{2} \left( \int_z^L Q(t) \, dt \right)^2 \,  \right]_{z=x}^{z=L} \\ &= \left( \int_x^L Q(w) P(w) \, dw \right)  \frac{1}{2} \left( \int_x^L Q(t) \, dt \right)^2. 
\end{align*}
Thus iterating $n-1$ times we get
$$0\le P(x)\le\frac{1}{n!}\left[\int_x^L Q(w)\,dw\right]^n
\int_x^L Q(z)P(z)\,dz.$$
Taking the limit as $n\to+\infty$, we get $P(x)=0$ and hence
$\bar{K}^{\up}(x,y)=\bar{K}^{\dn}(x,y)=0$ for
almost every $y>x$, as claimed. The proof of part (2) of Lemma 5.2 is analogous.

{\bf{Acknowledgements}}
The research has been partially supported by  INdAM (National Institute for Advanced Mathematics, Italy).

\bibliographystyle{amsplain}
\bibliography{biblio}

\end{document}